\documentclass[reqno]{amsart}	
\usepackage[pdftex]{graphicx}				
\usepackage{amsthm}
\usepackage{amsmath}
\usepackage{amsfonts}
\usepackage{amssymb}
\usepackage{ascmac}
\usepackage{tikz}
\usepackage{latexsym}
\usepackage{enumerate}

\usetikzlibrary{arrows.meta}

\theoremstyle{definition}
\newtheorem{Def}{Definition}[section]
\newtheorem{rem}[Def]{Remark}
\newtheorem{nota}[Def]{Notation}
\newtheorem{ex}[Def]{Example}

\theoremstyle{plain}
\newtheorem{prop}[Def]{Proposition}
\newtheorem{thm}[Def]{Theorem}
\newtheorem{lem}[Def]{Lemma}

\newcommand{\N}{{\mathbb N}}

\newcommand{\R}{{\mathbb R}}

\newcommand{\Span}{\mathrm{Span}}

\newcommand{\rank}{\mathop{\mathrm{rank}}}

\def\i<#1>{\langle #1 \rangle}
\def\l<#1>{\left\langle #1 \right\rangle}

\newcommand{\Ch}{\mathrm{Ch}}
\newcommand{\DD}{\partial}

\newcommand{\calI}{\mathcal{I}}
\newcommand{\Sing}{\mathrm{Sing}}

\newcommand{\Ker}{\mathop{\mathrm{Ker}}}

\newcommand{\corank}{\mathop{\mathrm{corank}}}

\newcommand{\pr}{\mathrm{pr}}

\makeatletter
  
  \@addtoreset{equation}{section}
\makeatother

\allowdisplaybreaks[4]

\makeatletter
\def\keywords{\xdef\@thefnmark{}\@footnotetext}
\makeatother

\title{Second-order generalized Monge--Amp\`ere equations on a plane and its geometric singular solutions}

\author{Masahiro Kawamata}

\begin{document}
\tikzset{auto}
\pagestyle{empty}


\maketitle

\begin{abstract}
In the present paper, we study some generalized Monge--Amp\`ere equations in terms of special exterior differential systems on a jet space. Moreover, we construct geometric singular solutions of the generalized Monge--Amp\`ere equations by using the method of Cauchy characteristics. Furthermore, we give criteria for geometric singular solutions to be right-left equivalent to the cuspidal edge, swallowtail and butterfly.

\end{abstract} 


\section{Introduction}
In recent years, differential equations have been studied not only by analytical, but also by geometrical methods. 
Among them, there are many studies of differential equations using differential ideals in the ring of differential forms on manifolds (called exterior differential systems, see Definition~\ref{def:EDS}). According to the theory of exterior differential systems, $k$th-order differential equations can usually be analyzed by the pair of submanifolds in $k$-jet spaces and canonical exterior differential systems on them. Then, symmetries, classifications and quadratures of differential equations have been investigated actively from the viewpoint of exterior differential systems (\cite{bryant, bryant2, ishikawa-morimoto, KLR, LRC, morimoto95}). On the other hand, one of the most important class of differential equations is the Monge-Amp\`ere equation (or MAE for short), which is a class of second-order partial differential equations with two independent variables. This class of partial differential equations has appeared in mathematics, physics, and many other fields. Under the above situation, MAEs can usually be investigated on 2-jet spaces. However, Morimoto defines a special exterior differential system on a 1-jet space (called Monge--Amp\`ere systems, or MAS for short) for each MAE which is a second-order partial differential equation, and show that there is a natural correspondence between (local) solutions of MAEs and (local) integral manifolds of MASs (\cite{morimoto95}). Namely, Morimoto gives a new method to analyze MAEs in the sense of the theory of exterior differential systems.

As a study of MAEs from another point of view, generalization problems have been considered. For example, the symplectic Monge--Amp\`ere equation and the third-order Monge--Amp\`ere equation are well-known generalizations of MAEs (\cite{Boillat, Boillat2, donate-valenti}). In our previous study, the author and Shibuya define the \textit{generalized Monge--Amp\`ere equation} (or \textit{GMAE} for short), which includes the above -mentioned generalizations of MAEs and other important differential equations (\cite{KS}). Moreover, there is a natural correspondence between (local) solutions of $k$th order GMAEs and (local) integral manifolds of special exterior differential systems on $(k-1)$-jet spaces (called \textit{generalized Monge--Amp\`ere systems}, or \textit{GMAS} for short). Therefore, it is expected that partial differential equations included in GMAEs which are not included in MAEs (called a \textit{non-trivial GMAE}) are treated like MAEs. However, there are many properties of GMAEs and GMASs which have not yet been examined.

On the other hand, solutions of differential equations with graphs which admit singularities have been also studied, and such solutions are called \textit{geometric singular solutions} or \textit{geometric solutions} (see Definition~\ref{def:singular-sol}). In fact, Izumiya and Kossioris construct geometric singular solutions of the Hamilton-Jacobi equation and classify their singularities (\cite{izumiya1993, izumiya-kossioris1995, izumiya-kossioris1997}). As for the differential equations of Clairaut type, Saji and Takahashi investigate the relationship between its geometric singular solutions and envelopes obtained from them (\cite{takahashi-saji}). For MAEs, Morimoto and Ishikawa show that a well-known singularity (called the open umbrella) appears as a geometric singular solution (\cite{ishikawa-morimoto}).

In this paper, we examine the following system of partial differential equations which is one of the simplest non-trivial GMAE by using the corresponding GMAS, and give some properties in terms of exterior differential systems:
\begin{flalign}\label{ODS000}
	\left\{
		\begin{array}{l}
			z_{xx}=\alpha(x,y,z,z_{x},z_{y}) z_{xy}\\
			z_{xy}=\alpha(x,y,z,z_{x},z_{y})z_{yy}.
			\end{array}
	\right.
\end{flalign}
In particular, we obtain first integrals of the Cauchy characteristic system of the GMAS using the function $\alpha$, explicitly. Then, for the case that the Cauchy characteristic system of the GMAS does not vanish, we classify the GMAE into several classes. In addition, we give criteria for the GMAE included in each such class. Moreover, we explicitly construct geometric singular solutions of the GMAE (\ref{ODS000}) by using the method of Cauchy characteristics and investigate the obtained geometric singular solutions by applying the singularity theory (\cite{USY}). Especially, we give criteria for these geometric singular solutions to be right-left equivalent to the cuspidal edge, swallowtail and butterfly.

Our paper is organized as follows. In Section 2, we recall the basics of exterior differential systems and jet spaces. In Section 3, we give some basic properties of a GMAE from the viewpoint of exterior differential systems. In Section 4, we calculate first integrals of the Cauchy characteristic system of the GMAS corresponding to (\ref{ODS000}). In Section 5, we classify GMASs into several classes using the theory of Pfaffian systems. In Section 6, we construct geometric singular solutions of (\ref{ODS000}) and give criteria for the obtained solutions to be right-left equivalent to the cuspidal edge, swallowtail and butterfly. Furthermore, we also construct concrete examples of geometric singular solutions which have these singularities.

In the present paper, we assume all objects are of class $C^{\infty}$.

\section{Preliminaries}
In this section, we recall the foundations of exterior differential systems. Let $\Omega^{\ast}(M)$ and $\Omega^{i}(M)$ be the set of all differential forms and all differential $i$-forms, respectively, on a manifold $M$. 

\begin{Def}\label{def:EDS}
	A subset $\mathcal{J}$ of $\Omega^{\ast}(M)$ is called an \textit{exterior differential system} (\textit{EDS} for short) on a manifold $M$, if $\mathcal{J}$ satisfies the following conditions: 
	\begin{enumerate}
  		\item The subset $\mathcal{J}$ is a homogeneous ideal of $\Omega^{\ast}(M)$, that is, $\mathcal{J}$ is written by the following form:
		\[
			\mathcal{J} = \bigoplus^{\infty}_{i=0}(\mathcal{J} \cap \Omega^{i}(M)).
		\]
  		\item The subset $\mathcal{J}$ is closed under exterior differentiation on $M$.
 	\end{enumerate}
\end{Def}

	Let $\{\omega_{1},\ldots,\omega_{r}\}_{\mathrm{alg}}$ denote the ideal in $\Omega^{\ast}(M)$ generated by homogeneous differential forms $\omega_{1},\ldots,\omega_{r}$ $\in \Omega^{\ast}(M)$. Moreover, we define the ideal $\{\omega_{1},\ldots,$ $\omega_{r}\}_{\mathrm{diff}}$ by 
	\[
		\{\omega_{1},\ldots,\omega_{r}\}_{\mathrm{diff}} := \{\omega_{1},\ldots,\omega_{r},d\omega_{1},\ldots,d\omega_{r}\}_{\mathrm{alg}},
	\]
	where $d \omega_{i}$ is the exterior derivative of $\omega_{i}$. Then the ideal $\{\omega_{1},\ldots,\omega_{r}\}_{\mathrm{diff}}$ is an EDS on $M$, and called \textit{the EDS generated by $\omega_{1},\ldots,\omega_{r}$}.

We next define a jet space. In this paper, we only define a 1-jet space on some three-dimensional manifolds. Hereafter, let $J^{0}(2,1)$ be a three-dimensional manifold. We remark that the projective cotangent bundle $PT^{\ast}J^{0}(2,1)$ of $J^{0}(2,1)$ has a natural contact structure, that is, there exists natural contact form $\omega_{0}$ on $PT^{\ast}J^{0}(2,1)$.

\begin{Def}
	Let $J^{0}(2,1)$ be a three-dimensional manifold, and $\omega_{0}$ be the natural contact form on $PT^{\ast}J^{0}(2,1)$. Then we denote $PT^{\ast}J^{0}(2,1)$ by $J^{1}(2,1)$, and $J^{1}(2,1)$ equipped with the EDS $\{\omega_{0}\}_{\mathrm{diff}}$ on $J^{1}(2,1)$ is called the \textit{1-jet space}. 
	\end{Def}

\begin{rem}
	By Darboux's theorem, the above-mentioned contact form $\omega_{0}$ satisfies the following: for each point of $J^{1}(2,1)$, there exists a local coordinate system $(U;x,y,z,p,q)$ around this point such that
	\[
		\omega_{0}|_{U}=dz-pdx-qdy.
	\]
	Hereafter, we call such a local coordinate system a \textit{canonical coordinate system}.
\end{rem}

\begin{Def}
	Let $\mathcal{J}$ be an EDS on a manifold $M$. Then the set
	\[
		\Ch(\mathcal{J}):=\{X \in TM\ |\ \iota_{X}(\mathcal{J}) \subset \mathcal{J}\}
	\]
	is called the \textit{Cauchy characteristic system} of $\mathcal{J}$. Here, $\iota_{X}$ is the interior product.
\end{Def}

\begin{rem}\label{rem:cauchy}
	In general, the Cauchy characteristic system is not a subbundle of the tangent bundle. However, it is well-known that if the Cauchy characteristic system is a subbundle of the tangent bundle, then it is a completely integrable system. Moreover, the original EDS is reduced to an EDS defined on the leaf space determined by the Cauchy characteristic system (see Theorem~\ref{thm:reduction-EDS}).
\end{rem}

\section{Properties of GMAS}

In this section, we explain some properties of the GMAE and the GMAS which used in this paper. Throughout this section, we consider the following GMAE:
\begin{flalign}\label{GMAE-1}
	\left\{
		\begin{array}{l}
			Az_{xx} + Bz_{xy}+C=0 \\
			Az_{xy}+Bz_{yy}+D=0,\\
		\end{array}
	\right.
\end{flalign}
where $z=z(x,y)$ is an unknown function and $A,B,C$ and $D$ are given functions of $x,y,z,z_{x}$ and $z_{y}$. By Example~{4.2} of \cite{KS},
the GMAE (\ref{GMAE-1}) corresponds to the GMAS generated by the following 1-forms:

\begin{flalign}\label{GMAS-1}
	\omega_{0}:=dz-pdx-qdy,\hspace{4mm}\Psi\equiv Adp+Bdq+Cdx+Ddy \mod \omega_{0}.
\end{flalign}
Here $(x,y,z,p,q)$ is a canonical coordinate system of $J^{1}(2,1)$.

Hereafter, we assume that $\Psi \not \equiv 0 \mod \omega_{0}$, because if $\Psi \equiv 0 \mod \omega_{0} $, then $\calI$ coincides with the EDS $\{\omega_{0}\}_{\mathrm{
diff}}$. Moreover, we fix a canonical coordinate system and denote this coordinate system by $(x,y,z,p,q)$.

\begin{lem}\label{lem:factor}
	For the GMAS $\calI$, there exist an open subset $U$ of $J^{1}(2,1)$ and  1-forms $\eta_{1}, \eta_{2}$ defined on $U$ such that $\omega_{0}$, $\Psi$, $\eta_{1}$ and $\eta_{2}$ are linearly independent, and $d \omega_{0} \equiv  \eta_1 \wedge \eta_2 \mod \omega_{0},\Psi$ holds. In particular, we can take $\eta_{1},\eta_{2}$ as the following:
	\begin{enumerate}[\normalfont(1)]
		\setlength{\parskip}{0mm} 
  		\setlength{\itemsep}{0mm} 
		\setlength{\leftskip}{5mm}
		\item If $A\ne 0$ then $U=\{p\in J^{1}(2,1)\ |\ A(p) \ne 0\}$ and 
		\[
			\eta_1=dy-(B/A)dx,\ \eta_2=dq+(D/A)dx.
		\]
		\item If $B \ne 0$ then $U=\{p\in J^{1}(2,1)\ |\ B(p) \ne 0\}$ and 
		\[
			\eta_1=dx-(A/B)dy,\ \eta_2=dp+(C/B)dy.
		\]
		\item If $C\ne0$ then $U=\{p\in J^{1}(2,1)\ |\ C(p) \ne 0\}$ and 
		\[
			\eta_1=dy+(B/C)dp,\ \eta_2=dq-(D/C)dp.
		\]
		\item If $D\ne0$ then $U=\{p\in J^{1}(2,1)\ |\ D(p) \ne 0\}$ and 
		\[
			\eta_1=dx+(A/D)dq,\ \eta_2=dp-(C/D)dq.
		\]
	\end{enumerate}
\end{lem}
\begin{proof}
	By $\Psi \equiv Adp +Bdq +Cdx+Ddy \not \equiv  0 \mod \omega_{0}$, we divide into the following four cases:
	\[
		(1)\ A\ne 0,\hspace{3mm}(2)\ B\ne 0,\hspace{3mm}(3)\ C\ne 0,\hspace{3mm}(4)\ D\ne 0.
	\]
	Since it can be proved in the same way, we have only to show the case (1). Then we have
	\[
		dp \equiv -(1/A)(Bdq+Cdx+Ddy) \mod \omega_{0}, \Psi,
	\]
	and thus, we obtain
	\begin{flalign*}
		d\omega_{0} &= dx \wedge dp + dy \wedge dq \\
		& \equiv -(1/A)dx \wedge (Bdq +Cdx+Ddy) + dy \wedge dq  \mod \omega_{0},\Psi\\
		&\equiv (dy-(B/A)dx)\wedge dq -(D/A)dx \wedge (dy-(B/A)dx)  \mod \omega_{0},\Psi\\
		& \equiv (dy-(B/A)dx) \wedge (dq+(D/A)dx) \mod \omega_{0},\Psi.
	\end{flalign*}
	Hence, we can take $U$, $\eta_{1}$ and $\eta_{2}$ as
	\[
		U:=\{p \in J^{1}(2,1)\ |\ A(p) \ne 0\},\hspace{3mm}\eta_1 := dy-(B/A)dx,\hspace{3mm} \eta_2 := dq+(D/A)dx,
	\]
	which completes the proof.
\end{proof}

\begin{lem}\label{lem:dimension-cauchy}
	For the GMAS $\calI$ and the open set $U$ in Lemma~\ref{lem:factor}, the dimension of the Cauchy characteristic system of $\calI$ is one or zero.
\end{lem}

\begin{proof}
	By Lemma~\ref{lem:factor}, there exists a 1-form $\theta$ such that $\{\omega_{0},\Psi,\eta_{1},\eta_{2},\theta\}$ is a coframe of $T^{\ast}U$. Then we take the dual frame $\{X,Y,Z_{1},Z_{2},W\}$ of $\{\omega_{0},\Psi,\eta_{1},\eta_{2},\theta\}$. 
	
	For $V \in \Ch(\calI|_{U})$, we put
\[
	V=aX+bY+c_1Z_1+c_2Z_2+eW,
\]
where $a,b,c_{1},c_{2}$ and $e$ are functions on $U$. By the definition of $\Ch(\calI|_{U})$, we have $\iota_V(\omega_{0}), \iota_V(\Psi) \in \calI|_{U}$. Therefore, one has $\iota_V(\omega_{0})=\iota_V(\Psi)=0$, and thus, this yields $a=b=0$. Moreover, by $\iota_V(d\omega_{0}) \equiv 0 \mod \omega_{0},\Psi$, we obtain
\[
	\iota_V(d\omega_{0}) \equiv \iota_V(\eta_1 \wedge \eta_2) \equiv c_1\eta_2 - c_2\eta_1 \equiv 0 \mod \omega_{0},\Psi,
\]
then we have $c_1=c_2=0$. Hence it must be $V=eW$. Since $V$ is an arbitrary element of $\Ch(\calI|_{U})$, we have $\Ch(\calI|_{U}) \subset \Span\{W\}$. This completes the proof.
\end{proof}

\begin{prop}\label{prop:ch=1}
	For the GMAS $\calI$ and the open set $U$ in Lemma~\ref{lem:factor}, $\dim \Ch(\calI|_{U})\ne 0$ holds if and only if $\Psi$ satisfies $d\Psi \equiv 0 \mod \omega_{0}, d\omega_{0}, \Psi$ on $U$. 
\end{prop}

\begin{proof}
	As in the proof of Lemma~\ref{lem:dimension-cauchy}, we take a 1-form $\theta$ and the dual frame $\{X,Y,Z_{1},Z_{2},W\}$ of $\{\omega_{0},\Psi,\eta_{1},\eta_{2},\theta\}$. Then $d\Psi$ is written by 
	\[
	d\Psi \equiv a_1 \eta_1 \wedge \theta + a_2 \eta_2 \wedge \theta + a_{3} \eta_{1} \wedge \eta_{2} \mod \omega_{0},\Psi, 
	\]
	where $a_{1}, a_{2}$ and $a_{3}$ are functions on $U$. By Lemma~\ref{lem:factor}, we remark that
\begin{flalign}\label{dPsi-structure}
	d\Psi \equiv a_1 \eta_1 \wedge \theta + a_2 \eta_2 \wedge \theta \mod \omega_{0},\Psi,d\omega_{0}.
\end{flalign}

Under the above preparation, let us assume that $\dim \Ch(\calI|_{U}) \ne 0$. Then there exists a non-zero vector field $V$ on $U$ such that $V \in \Ch(\calI|_{U})$. By the definition of Cauchy characteristic system, we have $\iota_V(\omega_{0}),\iota_V(\Psi),\iota_V(d\omega_{0}), \iota(d\Psi) \in \calI|_{U}$. Since $\omega_{0}$ and $\Psi$ are 1-forms, one can easily see that 
\[
	\omega_{0}(V)=\Psi(V)=0.
\]
Moreover, by $d\omega_{0} \equiv \eta_{1}\wedge \eta_{2} \mod \omega_{0},\Psi$ and $\iota_{V}(d\omega_{0}) \in \calI|_{U}$, we have
\[
	\eta_{1}(V)\eta_{2} -\eta_{2}(V)\eta_{1} \equiv 0 \mod \omega_{0},\Psi.
\]
Then, since $\omega_{0},\Psi,\eta_{1}$ and $\eta_{2}$ are linearly independent, one has
\[
	\eta_{1}(V)=\eta_{2}(V)=0.
\]
Therefore, by $\dim \Ch(\calI|_{U})=1$, we obtain $\theta(V) \ne 0$. On the other hand, by $\iota_{V}(d\Psi) \in \calI$ and (\ref{dPsi-structure}), one can see that
\[
	-a_1\theta(V)\eta_1 -a_2 \theta(V)\eta_{2}  \equiv 0 \mod \omega_{0},\Psi.
\]
By $\theta(V) \ne 0$, we have $a_1=a_2=0$, which indicates $d\Psi \equiv 0 \mod \omega_{0},\Psi,d\omega_{0}$ on $U$.

	Conversely, we assume $d \Psi \equiv 0 \mod \omega_{0},\Psi,d\omega_{0}$ on $U$. We here show that $W \in \Ch(\calI|_{U})$. Since $\{X,Y,Z_{1},Z_{2},W\}$ is the dual frame of $\{\omega_{0},\Psi,\eta_{1},\eta_{2},\theta\}$, then we have
	\[
		\iota_{W}(\omega_{0})=\iota_{W}(\Psi)=0,\hspace{2mm}\iota_{W}(d\omega_{0}) \equiv \iota_{W}(\eta_{1}\wedge \eta_{2}) \equiv0 \mod \omega_{0},\Psi.
	\]
Moreover, by $d\Psi \equiv 0 \mod \omega_{0},\Psi,d\omega_{0}$, we obtain $\iota_{W}(d\Psi) \equiv 0 \mod \omega_{0},\Psi$ on $U$. By the definition of the Cauchy characteristic system, this yields $W \in \Ch(\calI|_{U})$, which completes the proof.
\end{proof}

\begin{rem}\label{rem:cauchy-dimension}
	By Lemma~\ref{lem:dimension-cauchy} and Proposition~\ref{prop:ch=1}, we remark that $\dim \Ch(\calI|_{U})=1$ if and only if $d \Psi \equiv 0 \mod \omega_{0},\Psi,d\omega_{0}$.
\end{rem}

\begin{nota}\label{notation:derivative}
	Let $f\colon J^{1}(2,1) \to \R$ be a function. Then
	\begin{enumerate}[\normalfont (1)]
		\item if $A \ne 0$, we put
			\[
				f_{x,A}:=\frac{df}{dx}-\frac{C}{A}f_{p},\hspace{2mm} f_{y,A}:=\frac{df}{dy} - \frac{D}{A}f_{p},\hspace{2mm} f_{q,A}:=f_{q}-\frac{B}{A}f_{p},
			\]
		\item if $B \ne 0$, we put
			\[
				f_{x,B}:=\frac{df}{dx}-\frac{C}{B}f_{q},\hspace{2mm} f_{y,B}:=\frac{df}{dy} - \frac{D}{B}f_{q},\hspace{2mm} f_{p,B}:=f_{p}-\frac{A}{B}f_{q},
			\]
		\item if $C \ne 0$, we put
			\[
				f_{y,C}:=\frac{df}{dy}-\frac{D}{C}\frac{df}{dx},\hspace{2mm} f_{p,C}:=f_{p} - \frac{A}{C}\frac{df}{dx},\hspace{2mm} f_{q,C}:=f_{q}-\frac{B}{C}\frac{df}{dx},
			\]
		\item if $D \ne 0$, we put
			\[
				f_{x,D}:=\frac{df}{dx}-\frac{C}{D}\frac{df}{dy},\hspace{2mm} f_{p,D}:=f_{p} - \frac{A}{D}\frac{df}{dy},\hspace{2mm} f_{q,D}:=f_{q}-\frac{B}{D}\frac{df}{dy}.
			\]
	\end{enumerate}
	Here, $d/dx$ and $d/dy$ are defined as follows:
	\[
		\frac{d}{dx}:=\frac{\DD}{\DD x} + p\frac{\DD}{\DD z},\hspace{5mm}\frac{d}{dy}:=\frac{\DD}{\DD x}+q\frac{\DD}{\DD z}.
	\]
\end{nota}

\begin{lem}\label{lem:calculate-df}
	Let $f \colon J^{1}(2,1) \to \R$ be a function. Then the following holds:
	\begin{enumerate}[\normalfont (1)]
		\item If $A \ne 0$, then $df \equiv f_{x,A}dx + f_{y,A}dy + f_{q,A}dq \mod \omega_{0},\Psi$.
		\item If $B \ne 0$, then $df \equiv f_{x,B}dx + f_{y,B}dy + f_{p,B}dp \mod \omega_{0},\Psi$.
		\item If $C \ne 0$, then $df \equiv f_{y,C}dy + f_{p,C}dp + f_{q,C}dq \mod \omega_{0},\Psi$. 
		\item If $D \ne 0$, then $df \equiv f_{x,D}dx + f_{p,D}dp + f_{q,D}dq \mod \omega_{0},\Psi$.
	\end{enumerate}
\end{lem}

\begin{proof}
	For the case of (2) or (4), by applying the following contact transformation to $\calI$, we arrive at the case of (1) or (3).
	\[
		J^{1}(2,1) \to J^{1}(2,1); (x,y,z,p,q) \mapsto (y,x,z,q,p).
	\]
	Therefore, we prove only cases (1) and (3), but only (1) will be shown because (3) can be proved by the same calculation as (1). 
	
	By $A \ne 0$, we have
	\begin{flalign*}
		dp \equiv -(1/A)(Bdq + Cdx + Ddy) \mod \omega_{0},\Psi.
	\end{flalign*}
	We remark that
	\begin{flalign*}
		dz \equiv pdx+qdy \mod \omega_{0}.
	\end{flalign*}
	Therefore, one has
	\begin{flalign*}
		df &=f_{x}dx+f_{y}dy+f_{z}dz+f_{p}dp+f_{q}dq \\
		&\equiv f_{x}dx+f_{y}dy+f_{z}(pdx+qdy) \\
		&\hspace{20mm}-(1/A)f_{p}(Bdq + Cdx + Ddy)+f_{q}dq \mod \omega_{0},\Psi \\
		&\equiv f_{x,A}dx + f_{y,A}dy + f_{q,A}dq \mod \omega_{0},\Psi.
	\end{flalign*}
	This completes the proof.
\end{proof}

\begin{prop}\label{prop:Cauchy-dim1-criterion}
	Let $\calI$ be the GMAS generated by $\omega_{0}$ and $\Psi$, where $\Psi \equiv Adp+Bdq+Cdx+Ddy \mod \omega_{0}$, and $U$ be the open set mentioned in Lemma~\ref{lem:factor}. For a point $p_{0} \in U$, the followings hold:
	\begin{enumerate}[\normalfont (1)]
		\item let us assume $A \ne 0$. Then $\dim \Ch(\calI)=1$ at $p_{0}$ if and only if the functions $A, B, C$ and $D$ satisfy the following condition at $p_{0}$:
			{\small
			\begin{flalign*}
				\left\{
					\begin{array}{l}
						\displaystyle B_{x,A}-C_{q,A}-\frac{1}{A}(A_{x,A}B-CA_{q,A})+\frac{B}{A}\left(B_{y,A}-D_{q,A}-\frac{1}{A}(A_{y,A}B-DA_{q,A})\right)\\
						\hspace{110mm}=0\\
						\displaystyle D_{x,A}-C_{y,A}-\frac{1}{A}(A_{x,A}D-A_{y,A}C)+\frac{D}{A}\left(B_{y,A}-D_{q,A}-\frac{1}{A}(A_{y,A}B-DA_{q,A})\right) \\
						\hspace{110mm}=0.
					\end{array}
				\right.
			\end{flalign*}
			}
			
			\noindent
			Conversely, $\dim \Ch(\calI)=0$ at $p_{0}$ if and only if the above condition does not hold at $p_{0}$.
		\item let us assume $B \ne 0$. Then $\dim \Ch(\calI)=1$ at $p_{0}$ if and only if the functions $A, B, C$ and $D$ satisfy the following condition at $p_{0}$:
		{\small
			\begin{flalign*}
				\left\{
					\begin{array}{l}
						\displaystyle D_{x,B}-C_{y,B}-\frac{1}{B}(B_{x,B}D-B_{y,B}C)-\frac{C}{B}\left(A_{x,B}-C_{p,B}-\frac{1}{B}(B_{x,B}A-CB_{p,B})\right) \\
						\hspace{110mm}=0\\
						\displaystyle A_{y,B}-D_{p,B}-\frac{1}{B}(B_{y,B}A-DB_{p,B})+\frac{A}{B}\left(A_{x,B}-C_{p,B}-\frac{1}{B}(B_{x,B}A-CB_{p,B})\right)\\
						\hspace{110mm}=0.\\
					\end{array}
				\right.
			\end{flalign*}
		}
		
		\noindent
		Conversely, $\dim \Ch(\calI)=0$ at $p_{0}$ if and only if the above condition does not hold at $p_{0}$.
		\item let us assume $C \ne 0$. Then $\dim \Ch(\calI)=1$ at $p_{0}$ if and only if the functions $A, B, C$ and $D$ satisfy the following condition at $p_{0}$:
		{\small
			\begin{flalign*}
				\left\{
					\begin{array}{l}
						\displaystyle A_{y,C}-D_{p,C}-\frac{1}{C}(C_{y,C}A-C_{p,C}D)+\frac{D}{C}\left(B_{y,C}-C_{q,C}-\frac{1}{C}(C_{y,C}B-C_{q,C}D)\right)\\
						\hspace{110mm}=0\\
						\displaystyle B_{p,C}-A_{q,C}-\frac{1}{C}(C_{p,C}B-C_{q,C}A)-\frac{B}{C}\left(B_{y,C}-C_{q,C}-\frac{1}{C}(C_{y,C}B-C_{q,C}D)\right)\\
						\hspace{110mm}=0.\\
					\end{array}
				\right.
			\end{flalign*}
		}
		
		\noindent
		Conversely, $\dim \Ch(\calI)=0$ at $p_{0}$ if and only if the above condition does not hold at $p_{0}$.
		\item let us assume $D \ne 0$. Then $\dim \Ch(\calI)=1$ at $p_{0}$ if and only if the functions $A, B, C$ and $D$ satisfy the following condition at $p_{0}$:
		{\small
			\begin{flalign*}
				\left\{
					\begin{array}{l}
						\displaystyle B_{x,D}-C_{q,D}-\frac{1}{D}(D_{x,D}B-D_{q,D}C)+\frac{C}{D}\left(A_{x,D}-C_{p,D}-\frac{1}{D}(D_{x,D}A-D_{p,D}C)\right)\\
						\hspace{110mm}=0\\
						\displaystyle B_{p,D}-A_{q,D}-\frac{1}{D}(D_{p,D}B-D_{q,D}A)+\frac{A}{D}\left(A_{x,D}-C_{p,D}-\frac{1}{D}(D_{x,D}A-D_{p,D}C)\right)\\
						\hspace{110mm}=0.
					\end{array}
				\right.
			\end{flalign*}
		}
		
		\noindent
		Conversely, $\dim \Ch(\calI)=0$ at $p_{0}$ if and only if the above condition does not hold at $p_{0}$.
	\end{enumerate}
\end{prop}

\begin{proof}
	As in the case of Lemma~\ref{lem:calculate-df}, we show only the case of (1). Let us assume $\dim\Ch(\calI|_{U}) =1$. By our assumption and Remark~\ref{rem:cauchy-dimension}, we have
	\begin{flalign}\label{dPsi-0}
		d\Psi \equiv dA \wedge dp + dB\wedge dq + dC\wedge dx + dD\wedge dy \equiv 0 \mod \omega_{0}, \Psi, d\omega_{0}.
	\end{flalign}
	Then one has
	\begin{flalign}
		\label{dp}dp &\equiv -(1/A)(Bdq+Cdx+Ddy) \mod \Psi.
	\end{flalign}
	By using Lemma~\ref{lem:calculate-df} (1), we obtain
	\begin{flalign}
		\label{dA}dA &\equiv A_{x,A}dx + A_{y,A}dy+A_{q,A}dq \mod \omega_{0},\Psi,\\
		\label{dB}dB &\equiv B_{x,A}dx + B_{y,A}dy+B_{q,A}dq \mod \omega_{0},\Psi,\\
		\label{dC}dC &\equiv C_{x,A}dx + C_{y,A}dy+C_{q,A}dq \mod \omega_{0},\Psi,\\
	\label{dD}dD &\equiv D_{x,A}dx + D_{y,A}dy+D_{q,A}dq \mod \omega_{0},\Psi.
	\end{flalign}
	Therefore, by substituting $(\ref{dp}),\ldots, (\ref{dD})$ into (\ref{dPsi-0}), we have
	\begin{flalign*}
		0 &\equiv d\Psi \mod \omega_{0},d\omega_{0},\Psi\\
		 &\equiv  (A_{x,A}dx + A_{y,A}dy+A_{q,A}dq) \wedge dp + (B_{x,A}dx + B_{y,A}dy) \wedge dq \\
		&\hspace{4mm}+( C_{y,A}dy+C_{q,A}dq) \wedge dx + (D_{x,A}dx +D_{q,A}dq) \wedge dy \mod \omega_{0},\Psi,d\omega_{0} \\
		&\equiv (-(1/A)(A_{x,A}B-CA_{q,A})+B_{x,A}-C_{q,A}) dx \wedge dq\\
		&\hspace{4mm}+(-(1/A)(A_{y,A}B-DA_{q,A})+B_{y,A}-D_{q,A})dy\wedge dq \\
		&\hspace{4mm}+(-(1/A)(A_{x,A}D-A_{y,A}C)-C_{y,A}+D_{x,A})dx\wedge dy \mod \omega_{0},\Psi,d\omega_{0}\\
		&\equiv \mbox{$(B_{x,A}-C_{q,A}-(1/A)(A_{x,A}B-CA_{q,A})$}\\
		&\hspace{4mm}\mbox{$+(B/A)\left(B_{y,A}-D_{q,A}-(1/A)(A_{y,A}B-DA_{q,A})\right))dx\wedge dq$} \\
		&\hspace{7mm}\mbox{$+(D_{x,A}-C_{y,A}-(1/A)(A_{x,A}D-A_{y,A}C)$} \\
		&\hspace{5mm}\mbox{$+(D/A)\left(B_{y,A}-D_{q,A}-(1/A)(A_{y,A}B-DA_{q,A})\right)$})dx\wedge dy \mod \omega_{0},\Psi,d\omega_{0}.
	\end{flalign*}
	Since $\omega_{0}\wedge \Psi,d\omega_{0},dx\wedge dq,dx\wedge dy$ are linearly independent, this yields that
	{\small
			\begin{flalign*}
				\left\{
					\begin{array}{l}
						\displaystyle B_{x,A}-C_{q,A}-\frac{1}{A}(A_{x,A}B-CA_{q,A})+\frac{B}{A}\left(B_{y,A}-D_{q,A}-\frac{1}{A}(A_{y,A}B-DA_{q,A})\right)\\
						\hspace{110mm}=0\\
						\displaystyle D_{x,A}-C_{y,A}-\frac{1}{A}(A_{x,A}D-A_{y,A}C)+\frac{D}{A}\left(B_{y,A}-D_{q,A}-\frac{1}{A}(A_{y,A}B-DA_{q,A})\right)\\ 
						\hspace{110mm}=0.
					\end{array}
				\right.
			\end{flalign*}
			}
	The converse is also shown by direct calculations, which completes the proof.
\end{proof}

\section{Properties of the GMAE of involutive homogeneous type}
In previous section, we consider some properties of the GMAE (\ref{GMAE-1}) from the view point of exterior differential systems. Among others, the equations (\ref{GMAE-1}) which the second order terms do not vanish especially are important. Therefore, we consider the following GMAE:
\begin{flalign}\label{ODS1}
	\left\{
		\begin{array}{l}
			z_{xx}=\alpha(x,y,z,z_{x},z_{y}) z_{xy}\\
			z_{xy}=\alpha(x,y,z,z_{x},z_{y})z_{yy}.
			\end{array}
	\right.
\end{flalign}
Let $(x,y,z,p,q)$ be a canonical coordinate system around some point of $J^{1}(2,1)$. Then, the GMAS $\calI$ corresponding to (\ref{ODS1}) is generated by the following 1-forms on that coordinate:
	\[
		\omega_{0}:=dz-pdx-qdy,\hspace{2mm}\Psi\equiv dp-\alpha(x,y,z,p,q) dq \mod \omega_{0}.
	\]
Hereafter, we consider some geometric properties of this GMAS $\calI$. Firstly, we give a vanishing condition for $\Ch(\calI)$.

\begin{prop}\label{prop:cauchy-alpha}
	For the GMAS $\calI$ and $p_{0} \in J^{1}(2,1)$, the followings hold:
	\begin{enumerate}[\normalfont (1)]
		\item the dimension of $\Ch(\calI)$ is one at $p_{0}$ if and only if the function $\alpha$ satisfies
		\[
			\alpha_{x}+p\alpha_{z}-\alpha(\alpha_{y}+q\alpha_{z})=0 \hspace{4mm}\text{at $p_{0}$}.
		\]
		\item the dimension of $\Ch(\calI)$ is zero at $p_{0}$ if and only if the function $\alpha$ satisfies
		\[
			\alpha_{x}+p\alpha_{z}-\alpha(\alpha_{y}+q\alpha_{z}) \ne 0 \hspace{4mm}\text{at $p_{0}$}.
		\]
	\end{enumerate}
\end{prop}

\begin{proof}
	By $A=1, B=-\alpha, C=D=0$, Proposition~\ref{prop:Cauchy-dim1-criterion} (1) and Notation~\ref{notation:derivative} (1), this proposition is proved immediately.
	\end{proof}

\begin{rem}
	There exists a function which satisfies the following at single point, but does not satisfy locally:
	\begin{flalign}\label{condition-alpha1}
		\alpha_{x}+p\alpha_{z}-\alpha(\alpha_{y}+q\alpha_{z})=0.
	\end{flalign}
	In the present paper, we do not deal with such functions (cf. Remark~\ref{rem:cauchy}).
\end{rem}

\begin{Def}
	The GMAE and GMAS determined by the function $\alpha$ satisfying (\ref{condition-alpha1}) locally are called \textit{involutive type}.
\end{Def}

\begin{ex}\label{ex:involutive-function}
	Let $f$ be a function of three independent variables, $g$ be a function of two independent variables and $C$ be a constant number. If $\alpha$ is one of the following three functions
	\[
		f(z-px-qy,p,q),\hspace{4mm}\frac{g(p,q)-y}{x},\hspace{4mm} -\frac{z}{qx} + \frac{C}{x} + \frac{p}{q},
	\]
	then the GMAS $\calI$ is of involutive type, that is, these functions are solutions of (\ref{condition-alpha1}).
\end{ex}

\subsection{First integrals of Cauchy characteristics}
In this subsection, we find first integrals of $\Ch(\calI)$. 

\begin{lem}\label{lem:definition-form-Ch}
	For the Cauchy characteristic system $\Ch(\calI)$, if $\dim \Ch(\calI)=1$ then the following holds:
	\[
		\Ch(\calI)=\{dy+\alpha dx=\omega_{0}=\Psi=dq=0\}.
	\]
\end{lem}

\begin{proof}
	Since the dimension of $\Ch(\calI)$ is one and Remark~\ref{rem:cauchy-dimension}, we have $d \Psi \equiv 0 \mod \omega_{0},d\omega_{0},\Psi$. This yields that
	\[
		\calI =\{\omega_{0},\Psi\}_{\mathrm{diff}} =\{\omega_{0},\Psi,d\omega_{0},d\Psi\}_{\mathrm{alg}}=\{\omega_{0},\Psi,d\omega_{0}\}_{\mathrm{alg}}.
	\]
	By Lemma~\ref{lem:factor} (1), we have
	\[
		d \omega_{0} \equiv (dy + \alpha dx) \wedge dq  \mod \omega_{0},\Psi,
	\]
	and $\omega_{0},\Psi,dy+\alpha dx$ and $dq$ are linearly independent. Then there exists a 1-form $\theta$ such that $\{\omega_{0},\Psi,dy+\alpha dx,dq,\theta\}$ is a coframe of $T^{\ast}J^{1}(2,1)$. Here we denote the dual frame of this coframe by $\{X,Y,Z_{1},Z_{2},W\}$. Then, since $W$ satisfies
	\[
		\omega_{0}(W)=\Psi(W)=(dy+\alpha dx)(W)=dq(W)=0,
	\] 
	we have $W \in \Ch(\calI)$. Hence, by $\dim \Ch(\calI)=1$, one has
	\[
		\Ch(\calI)=\Span\{W\}=\{ dy+\alpha dx =\omega_{0}=\Psi = dq =0\},
	\]
	which completes the proof.
\end{proof}

We next consider independent first integrals of $\Ch(\calI)$

\begin{prop}\label{prop:structure-Ch}
	For the Cauchy characteristics of $\calI$ and $p_{0} \in J^{1}(2,1)$, the followings hold:
	\begin{enumerate}[\normalfont (1)]
		\setlength{\parskip}{0mm} 
  		\setlength{\itemsep}{0mm} 
		\setlength{\leftskip}{0mm}
		\item \label{case1}
		 if $\alpha$ satisfies $1+x(\alpha_{y}+q\alpha_{z}) \ne0$ at $p_{0}$, then
		\[
			\Ch(\calI)=\{d(y+\alpha x)=d(z-px-qy)=dp=dq=0\} \hspace{4mm} \text{at $p_{0}$}.
		\]
		Namely, $y+\alpha x$, $z-px-qy$, $p$ and $q$ are independent first integrals of $\Ch(\calI)$.
		\item \label{case2}
			if $\alpha$ satisfies $1+x(\alpha_{y}+q\alpha_{z})=0$ at $p_{0}$, then 
		\[
			\Ch(\calI)=\{d\alpha=d(z-px-qy)=dp=dq=0\}\hspace{4mm} \text{at $p_{0}$}.
		\]
		Namely, $\alpha$, $z-px-qy$, $p$ and $q$ are independent first integrals of $\Ch(\calI)$.
	\end{enumerate}
\end{prop}

\begin{proof}
	For the both cases, we show that three independent first integrals can be taken as $z-px-qy$, $p$ and $q$. By Lemma~\ref{lem:definition-form-Ch}, we have
	\[
		\Ch(\calI) = \{ dy+\alpha dx =\omega_{0}=\Psi = dq =0\}.
	\]
	Therefore one can see that
	\[
		\Psi=dp-\alpha dq \equiv dp \mod dq,
	\]
	and thus we can take two independent first integrals of $\Ch(\calI)$ as $p$ and $q$. Moreover, we obtain
	\begin{flalign*}
		\omega_{0} &=dz-pdx-qdy = dz-d(px)+xdp -d(qy) +ydq \\
		&=d(z-px-qy)+xdp + ydq ,
	\end{flalign*}
	and thus, one has
	\[
		d(z-px-qy)=\omega_{0} -xdp - ydq.
	\]
	Hence, from the above discussion, the functions $z-px-qy$, $p$ and $q$ are independent first integrals of $\Ch(\calI)$.
	
	Firstly, we show (1). We assume that the function $\alpha$ satisfies $1+x(\alpha_{y}+q\alpha_{z}) \ne0$ at $p_{0}$. We put $\omega:=dy+\alpha dx$, then 
	\begin{flalign*}
		\omega&=dy+d(\alpha x) -xd\alpha \\
		&= d(y+\alpha x) -x(\alpha_{x}dx +\alpha_{y}dy + \alpha_{z}dz)-x\alpha_{p}dp -x\alpha_{q}dq \\
		&=d(y+\alpha x) -x(\alpha_{y}+q\alpha_{z})\omega-x\alpha_{z}\omega_{0}-x\alpha_{p}dp -x\alpha_{q}dq.
	\end{flalign*}
	Hence, we obtain
	\[
		d(y+\alpha x) = (1+x(\alpha_{y}+q\alpha_{z}))\omega + x\alpha_{z}\omega_{0}+x\alpha_{p}dp +x\alpha_{q}dq.
	\]
	Therefore, at the point $p_{0}$, the following holds:
	\begin{flalign}\label{matrix-transform1}
		\left(
			\begin{array}{c}
				d(y+\alpha x) \\
				d(z-px-qy) \\
				dp \\
				dq \\
			\end{array}
		\right) = \left(
			\begin{array}{cccc}
				1+x(\alpha_{y}+q\alpha_{z})& x\alpha_{z} &x\alpha_{p} & x\alpha_{q}\\
				0 & 1&-x & -y\\
				0& 0& 1& 0\\
				0& 0& 0& 1\\
			\end{array}
		\right)
		\left(
			\begin{array}{c}
				\omega \\
				\omega_{0} \\
				dp \\
				dq\\
			\end{array}
		\right).
	\end{flalign}
	Since $\alpha$ satisfies $1+x(\alpha_{y}+q\alpha_{z}) \ne 0$, the coefficient matrix of the right-hand side of (\ref{matrix-transform1}) is regular. Then, this yields that
	\[
		\Ch(\calI)=\{d(y+\alpha x)=d(z-px-qy)=dp=dq=0\} \hspace{4mm} \text{at $p_{0}$},
	\]
	that is, $y+\alpha x$, $z-px-qy$, $p$ and $q$ are independent first integrals of $\Ch(\calI)$.

	We next show (2). We assume that the function $\alpha$ satisfies 
	\begin{flalign}\label{condition-PDE}
		1+x(\alpha_{y}+q\alpha_{z})=0 \hspace{4mm} \text{at $p_{0}$}.
	\end{flalign}
	Then $\alpha$ also satisfies (\ref{condition-alpha1}), we have
	\[
		\alpha = -x(\alpha_{x}+p \alpha_{z}).
	\]
	The same as the proof of (1), we put $\omega:=dy + \alpha dx$ and one has
	\begin{flalign*}
		\omega &= dy - x(\alpha_{x}+p\alpha_{z})dx \\
			   &= dy-x(d\alpha -(\alpha_{y}+q\alpha_{z})dy -\alpha_{z}\omega_{0} -\alpha_{p}dp-\alpha_{q}dq) \\
			   &= -x d\alpha +x \alpha_{z}\omega_{0} + x\alpha_{p}dp + x\alpha_{q}dq.
	\end{flalign*}
	By the condition (\ref{condition-PDE}), it must be $x \ne 0$. Hence we obtain
	\[
		d\alpha = -(1/x) \omega +\alpha_{z}\omega_{0} + \alpha_{p}dp + \alpha_{q}dq.
	\]
	Therefore, we have
	\begin{flalign}\label{matrix-transform2}
		\left(
			\begin{array}{c}
				d\alpha \\
				d(z-px-qy) \\
				dp \\
				dq \\
			\end{array}
		\right) = \left(
			\begin{array}{cccc}
				-(1/x)&\alpha_{z} &\alpha_{p} &\alpha_{q} \\
				0 & 1&-x & -y\\
				0& 0& 1& 0\\
				0& 0& 0& 1\\
			\end{array}
		\right)
		\left(
			\begin{array}{c}
				\omega \\
				\omega_{0} \\
				dp \\
				dq\\
			\end{array}
		\right).
	\end{flalign}
	The coefficient matrix of the right-hand side of (\ref{matrix-transform2}) is regular. This yields that
	\[
		\Ch(\calI)=\{d\alpha=d(z-px-qy)=dp=dq=0\}.
	\]
	Hence, the functions $\alpha$, $z-px-qy$, $p$ and $q$ are independent first integrals. This completes the proof.
\end{proof}

\begin{rem}
	We remark that Proposition~\ref{prop:structure-Ch} holds even if $\alpha$ satisfies the assumptions at one point. Of course, if $\alpha$ locally satisfies the assumptions, then Proposition~\ref{prop:structure-Ch} holds locally.
\end{rem}

\begin{rem}
	There exists a function $\alpha$ which satisfies $1 + x(\alpha_{y}+q\alpha_{z})=0$ at a point, but not locally. We do not deal with such wild functions in this paper.
\end{rem}

\begin{Def}
	The function $\alpha$ which satisfies $1+x(\alpha_{y}+q\alpha_{z}) \ne 0$ is called \textit{generic type}. On the other hand, the function $\alpha$ which satisfies $1+x(\alpha_{y}+q\alpha_{z}) = 0$ locally is called \textit{non-generic type}.
\end{Def}

By Remark~\ref{rem:cauchy}, $\Ch(\calI)$ is completely integrable. This gives a one dimensional foliation of $J^{1}(2,1)$. We denote the leaf space with respect to this foliation by $J^{1}(2,1)/\Ch(\calI)$, and remark that $J^{1}(2,1)/\Ch(\calI)$ is locally a four dimensional manifold.

If $\alpha$ is of generic type, we put
\[
	x_{1}:=y+\alpha x,\ x_{2}:=z-px-qy,\ x_{3}:=p\ x_{4}:=q.
\]
On the other hand, if $\alpha$ is of non-generic type, we put
\[
	x_{1}:=\alpha,\ x_{2}:=z-px-qy,\ x_{3}:=p,\ x_{4}:=q.
\]
Thus $(x_{1},x_{2},x_{3},x_{4})$ is a local coordinate of $J^{1}(2,1)/\Ch(\calI)$, and the following map
\[
	\pi \colon J^{1}(2,1)\to J^{1}(2,1)/\Ch(\calI);\ (x,y,z,p,q) \mapsto  (x_{1},x_{2},x_{3},x_{4})
\]
gives the natural fibration.

\subsection{Generator of the reduced GMAS in generic type}
Throughout this subsection, we assume that the function $\alpha$ is of generic type. From a general theory of exterior differential systems, it is known that the following theorem:
\begin{thm}[\cite{BCG3,Cfb}]\label{thm:reduction-EDS}
	Let $\mathcal{J}$ be an EDS defined on a manifold $M$. If the leaf space $M/\Ch(\mathcal{J})$ defined by $\Ch(\mathcal{J})$ is a manifold, then there exists an EDS $\tilde{\mathcal{J}}$ defined on $M/\Ch(\mathcal{J})$ such that $\mathcal{J} = \{\pi^{\ast}\tilde{\mathcal{J}}\}_{\mathrm{alg}}$, where $\pi \colon M \to M/\Ch(\mathcal{J})$ be the natural projection.
\end{thm}

\noindent
Since the GMAS $\calI$ is generated by two 1-forms, we expect that there exist 1-forms $\xi_{1}$ and $\xi_{2}$ defined on $J^{1}(2,1)/\Ch(\calI)$ such that
\[
	\calI=\{\pi^{\ast}\xi_{1}, \pi^{\ast}\xi_{2},\pi^{\ast}(d\xi_{1}),\pi^{\ast}(d\xi_{2})\}_{\mathrm{alg}} = \{\pi^{\ast}\xi_{1}, \pi^{\ast}\xi_{2}\}_{\mathrm{diff}}.
\]
In this subsection, for a GMAS $\calI$ of generic type, we find such 1-forms $\xi_{1}$ and $\xi_{2}$ defined on $J^{1}(2,1)/\Ch(\calI)$.

We recall that $J^{1}(2,1)/\Ch(\calI)$ has a local coordinate $(U_{1};x_{1}, x_{2},x_{3},x_{4})$
defined by
\[
	x_{1}:=y+\alpha x,\ x_{2}:=z-px-qy,\ x_{3}:=p,\ x_{4}:=q.
\]
Then, for computational convenience, we decompose the natural fibration $\pi \colon J^{1}(2,1)$ $\to U_{1}$ as follows: we define the map  $\Phi_{1}\colon J^{1}(2,1) \to \R \times U_{1}$ by
\begin{flalign}\label{def-Phi1}
	x_{0}:=x,x_{1}:=y+\alpha x,\ x_{2}:=z-px-qy,\ x_{3}:=p,\ x_{4}:=q.
\end{flalign}
Then one can easily see that the following diagram commute:
\begin{center}
		\begin{tikzpicture}[auto]
			\node (J1) at (0, 0) {$J^{1}(2,1)$}; 
			\node (RU1) at (2.5,0) {$\R\times U_{1}$};
			\node (U1) at (2.5, -1.5) {$U_{1}$};
			\draw[->] (J1) to node {$\scriptstyle \Phi_{1}$} (RU1);
			\draw[->] (RU1) to node {$\scriptstyle \pr$} (U1);
			\draw[->] (J1) to node[swap] {$\scriptstyle \pi$} (U1);
			\end{tikzpicture}
		\end{center}
where $\pr$ is the natural projection.
Let $J_{\Phi_{1}}$ be the Jacobian of $\Phi_{1}$. By direct calculations, we have
\begin{flalign*}
	J_{\Phi_{1}}=\left(
		\begin{array}{ccccc}
			1& 0& 0& 0&0 \\
			\alpha_{x}x+\alpha& 1+\alpha_{y}x& \alpha_{z}x&\alpha_{p}x &\alpha_{q}x \\
			-p&-q &1 &-x &-y \\
			0&0 &0 &1 &0 \\
			0& 0&0 &0 &1 \\
		\end{array}
	\right),
\end{flalign*}
and thus, one has
\[
	\det J_{\Phi_{1}} = -(1+x(\alpha_{y}+q\alpha_{z})) \ne 0. 
\]
Therefore, by the inverse map theorem, there exists the locally inverse map $\Phi_{1}^{-1}$ defined on some open subset $V_{1}$ of $\R \times U_{1}$. 
\begin{nota}\label{notation:function-bar}
	In the following, we consider differential forms defined on $J^{1}(2,1)/$ $\Ch(\calI)$ corresponding to $\omega_{0}$ and $\Psi$. Then, we need to consider the functions which appear in the coefficients of $\omega_{0}$ and $\Psi$. Therefore, for this technical reasons, we introduce the following notation: let $f$ be a function defined on $\Phi^{-1}_{1}(V_{1})$. Then we denote the function on $V_{1}$ induced by $f$ as
	\[
		\overline{f}:=f \circ \Phi_{1}^{-1} \colon V_{1} \to \R.
	\]
\end{nota}

Here, we remark the following properties of this notation.

\begin{lem}
	 Let $f_{1}$ and $f_{2}$ be functions defined on $\Phi^{-1}_{1}(V_{1})$. Then the following holds:
	\begin{flalign}\label{property-bar1}
		\overline{f_{1}+f_{2}}=\overline{f_{1}}+\overline{f_{2}},\hspace{5mm}\overline{f_{1}f_{2}}=(\overline{f_{1}})(\overline{f_{2}}).
	\end{flalign}
\end{lem}

\begin{proof}
	This is obvious because it is a property of the composition of mappings. 
\end{proof}

The Jacobian of $\Phi^{-1}_{1}$ coincides with the inverse matrix of $J_{\Phi_{1}}$. Therefore, by using the Gaussian elimination, we have
{\small
\begin{flalign}\label{Jacobian1}
	J_{\Phi_{1}^{-1}}=
	\left(
		\begin{array}{ccccc}
			 \overline{1}& 0& 0&0 &0 \\
			-\overline{\alpha}& \overline{\rho}&-\overline{\rho x \alpha_{z}} & -\overline{\rho x (\alpha_{p}+x\alpha_{z})}&-\overline{\rho x (\alpha_{q}+y\alpha_{z})} \\
			\overline{p-q\alpha}& \overline{\rho q}&\overline{1-\rho q x\alpha_{z} }&\overline{x- \rho qx (\alpha_{p}+x\alpha_{z})}& \overline{y-\rho qx (\alpha_{q}+y\alpha_{z})} \\
			0&0 & 0&\overline{1} &0 \\
			0&0 &0 &0 &\overline{1} \\
		\end{array}
	\right),
\end{flalign}
}

\noindent
where $\rho:=1/(1+x(\alpha_{y}+q\alpha_{z}))$. Let the second and third components of $\Phi^{-1}_{1}$ be $\varphi_{1}$ and $\psi_{1}$, respectively. Then, by (\ref{Jacobian1}), we have
	{\small
	\begin{flalign}
		\label{dphi1}d\varphi_{1} &= -\overline{\alpha}dx_{0}+ \overline{\rho}dx_{1} -\overline{\rho x \alpha_{z}}dx_{2} - \overline{\rho x(\alpha_{p}+x\alpha_{z}}) dx_{3} - \overline{\rho x(\alpha_{q}+\varphi_{1}\alpha_{z}})dx_{4}, \\
		\label{dpsi1}
			\hspace{-2mm}d\psi_{1}&=\overline{p-\alpha q}dx_{0}+\overline{\rho q}dx_{1} + \overline{1-\rho q x\alpha_{z}}dx_{2} +\overline{x- \rho qx (\alpha_{p}+x\alpha_{z})}dx_{3}\\
			&\notag \hspace{73mm}+\overline{y-\rho qx (\alpha_{q}+y\alpha_{z})}dx_{4}.
	\end{flalign}
	}
We next show that the function $\overline{\alpha}$ is defined on $J^{1}(2,1)/\Ch(\calI)$.

\begin{lem}\label{lem:alpha-function}
	The partial derivative $\overline{\alpha}$ by $x_{0}$ is zero. Namely, $\overline{\alpha}$ is a function defined on $U_{1} ( \subset J^{1}(2,1)/\Ch(\calI))$.
\end{lem}

\begin{proof}
	Remark that 
	\[
	\overline{\alpha}(x_{0},x_{1},x_{2},x_{3},x_{4})=\alpha \circ \Phi^{-1}_{1}(x_{0},x_{1},x_{2},x_{3},x_{4}).
	\] 
	Then, by the chain rule and (\ref{property-bar1}), we have
	\begin{flalign*}
		(\overline{\alpha})_{x_{0}}&=\overline{\alpha_{x}}(x_{0})_{x_{0}} + \overline{\alpha_{y}}(\varphi_{1})_{x_{0}} + \overline{\alpha_{z}}(\psi_{1})_{x_{0}}\\
		&=\overline{\alpha_{x}}+\overline{\alpha_{y}}(\overline{-\alpha})+\overline{\alpha_{z}}(\overline{p-q\alpha})\\
		&=\overline{\alpha_{x} +p\alpha_{z} - \alpha (\alpha_{y}+q\alpha_{z})}.
	\end{flalign*}
	Since $\alpha$ satisfies the condition (\ref{condition-alpha1}), it must be $(\overline{\alpha})_{x_{0}} =0$. This completes the proof.
\end{proof}

\begin{prop}\label{prop:form-I/ChI-1}
	If the function $\alpha$ is of generic type, then the following holds:
	\[
		\calI=\left\{
			\pi^{\ast}(dx_{2}+x_{1}dx_{4}),\pi^{\ast}(dx_{3}-\overline{\alpha} dx_{4})
		\right\}_{\mathrm{diff}},
	\]	
	and thus we obtain
	\[
		\tilde{\calI}=\{dx_{2}+x_{1}dx_{4},dx_{3}-\overline{\alpha} dx_{4}\}_{\mathrm{diff}}.
	\]
\end{prop}

\begin{proof}
	In order to find the differential forms which generate $\tilde{\calI}$, we pull back $\omega_{0}$ and $\Psi$ to $V_{1}(\subset \R \times J^{1}(2,1)/\Ch(\calI))$. By (\ref{dphi1}) and (\ref{dpsi1}), we obtain
	\begin{flalign}
		\label{righthandside-defform1} (\Phi_{1}^{-1})^{\ast}\omega_{0}&=d\psi_{1} -x_{3}dx_{0}-x_{4}d\varphi_{1}=dx_{2}+x_{0}dx_{3}+\overline{y}dx_{4},\\
		\label{phi-inverse-form1}(\Phi_{1}^{-1})^{\ast}\Psi &= dx_{3}-\overline{\alpha}dx_{4}.
	\end{flalign}
	Then, by Lemma~\ref{lem:alpha-function}, the function $\overline{\alpha}$ is defined on $U_{1}$. Therefore, the right-hand side of (\ref{righthandside-defform1}) is a differential 1-form on $U_{1}$.
	On the other hand, since the right-hand side of (\ref{righthandside-defform1}) has a fiber coordinate $x_{0}$, it is not a 1-form on $U_{1}$. To eliminate $x_{0}$, we  perform the following calculation:
	\begin{flalign}
		\notag(\Phi_{1}^{-1})^{\ast}(\omega_{0}-x\Psi)&=dx_{2}+x_{0}dx_{3}+\overline{y}dx_{4}-x_{0}(dx_{3}-\overline{\alpha}dx_{4})\\
		\notag&=dx_{2}+\overline{y+\alpha x}dx_{4} \\
		&\label{defform-omega0-xPsi1}=dx_{2}+x_{1}dx_{4}.
	\end{flalign}
	Here, the right-hand side of (\ref{defform-omega0-xPsi1}) is a 1-form defined on $U_{1}$. 
	
	From the above discussion, we find a candidate for a generator of $\tilde{\calI}$. We next show that these are the desired 1-forms. Here, we remark that $\pi= \pr|_{V_{1}} \circ \Phi_{1}$. Then we have
	\begin{flalign*}
		\pi^{\ast}(dx_{2}+x_{1}dx_{4})&=(\Phi_{1}^{\ast}\circ (\pr|_{V_{1}})^{\ast})(dx_{2}+x_{1}dx_{4})=\Phi_{1}^{\ast} (\Phi_{1}^{-1})^{\ast}(\omega_{0}-x\Psi) \\
		&=\omega_{0}-x\Psi,\\
		\pi^{\ast}(dx_{3}-\overline{\alpha}dx_{4})&=(\Phi_{1}^{\ast}\circ (\pr|_{V_{1}})^{\ast})(dx_{3}-\overline{\alpha}dx_{4})=\Phi_{1}^{\ast}(\Phi_{1}^{-1})^{\ast}(\Psi)=\Psi.
	\end{flalign*}
	Therefore, one has
	\[
		\calI=\{\pi^{\ast}(dx_{2}+x_{1}dx_{4}),\pi^{\ast}(dx_{3}-\overline{\alpha}dx_{4})\}_{\mathrm{diff}},
	\]
	and thus
	\[
		\tilde{\calI}=\{dx_{2}+x_{1}dx_{4},dx_{3}-\overline{\alpha}dx_{4}\}_{\mathrm{diff}}.
	\]
	This completes the proof.
\end{proof}

\subsection{Generator of the reduced GMAS in non-generic type}
Throughout this subsection, we assume that the function $\alpha$ is of non-generic type. As in the previous subsection, for the GMAS $\calI$ of non-generic type, we find generators $\xi_{1}$ and $\xi_{2}$ of the reduced GMAS $\tilde{\calI}$. We recall that $J^{1}(2,1)/\Ch(\calI)$ has a local coordinate $(U_{2};x_{1},x_{2},x_{3},x_{4})$
defined by
\[
	x_{1}:=\alpha,\ x_{2}:=z-px-qy,\ x_{3}:=p\ x_{4}:=q.
\]
By the same reason as in the previous subsection, we decompose the natural fibration $\pi \colon J^{1}(2,1) \to U_{2}$ as follows: we define the map  $\Phi_{2}\colon J^{1}(2,1) \to \R \times U_{2}$ by
\begin{flalign}\label{def-Phi2}
	x_{0}:=x,\ x_{1}:=\alpha,\ x_{2}:=z-px-qy,\ x_{3}:=p,\ x_{4}:=q.
\end{flalign}
Then one can easily see that $\pi = \pr \circ \Phi_{2}$,
where $\pr \colon \R \times U_{2} \to U_{2}$ is the natural projection.
Let $J_{\Phi_{2}}$ be the Jacobian of $\Phi_{2}$. Then, by direct calculations, we have
\[
	J_{\Phi_{2}}=\left(
		\begin{array}{ccccc}
			1& 0& 0& 0&0 \\
			\alpha_{x}& \alpha_{y}& \alpha_{z}&\alpha_{p} &\alpha_{q} \\
			-p& -q& 1&-x & -y\\
			0& 0& 0&1 &0 \\
			0&0 &0 &0 &1 \\
		\end{array}
	\right).
\]
Since $\alpha$ satisfies $1 + x(\alpha_{y}+q\alpha_{z})=0$, one has
\[
	\det J_{\Phi_{2}} = \alpha_{y} + q\alpha_{z} = -(1/x) \ne 0.
\]
Therefore, by the inverse map theorem, there exists the locally inverse map $\Phi^{-1}_{2}$ defined on some open subset $V_{2}$ in $\R \times U_{2}$. 
\begin{nota}
	For the same technical reasons in Notation~\ref{notation:function-bar}, we introduce the following notation: let $f$ be a function defined on $\Phi^{-1}_{2}(V_{2})$. Then we denote the function on $V_{2}$ induced by $f$ as
	\[
		\overline{f}:=f \circ \Phi_{2}^{-1} \colon V_{2} \to \R.
	\]
\end{nota}
	Here, we remark the following properties of this notation.
\begin{lem}\label{lem:property-bar2}
	 Let $f_{1}$ and $f_{2}$ be functions defined on $\Phi^{-1}_{2}(V_{2})$. Then the following holds:
	 \begin{flalign}\label{property-bar1}
		\overline{f_{1}+f_{2}}=\overline{f_{1}}+\overline{f_{2}},\hspace{3mm}\overline{f_{1}f_{2}}=(\overline{f_{1}})(\overline{f_{2}}).
	\end{flalign}
\end{lem}

The Jacobian of $\Phi^{-1}_{2}$ coincides with the inverse matrix $J_{\Phi_{2}}$. Therefore, by using the Gaussian elimination, we have
{\small
\begin{flalign}\label{Jacobian2}
	J_{\Phi_{2}^{-1}}=\left(
		\begin{array}{ccccc}
			\overline{1}& 0& 0&0 &0 \\
			\overline{x(\alpha_{x}+p\alpha_{z})} & -\overline{x} &\overline{x\alpha_{z}} &\overline{x(\alpha_{p}+x\alpha_{z})} &\overline{x(\alpha_{q}+y\alpha_{z})}  \\
			\overline{p+qx(\alpha_{x}+p\alpha_{z})}&-\overline{qx} & \overline{1+qx\alpha_{z}}& \overline{x(1+q(\alpha_{p}+x\alpha_{z}))} & \overline{y+qx(\alpha_{q}+y\alpha_{z})}\\
			0&0 &0 &\overline{1} & 0\\
			0&0 &0 &0 &\overline{1} \\
		\end{array}
	\right).
\end{flalign}
}
Let the second and third components of $\Phi^{-1}_{2}$ be $\phi_{2}$ and $\psi_{2}$, respectively. Then, by (\ref{Jacobian2}), we have
{\small
\begin{flalign}
		\label{dphi2}&d\varphi_{2}=\overline{x(\alpha_{x}+p\alpha_{z})}dx_{0}-\overline{x} dx_{1}+\overline{x\alpha_{z}}dx_{2} + \overline{x(\alpha_{p}+x\alpha_{z})}dx_{3} + \overline{x(\alpha_{q}+y\alpha_{z})}dx_{4},  \\
		\label{dpsi2}
				&d\psi_{2}=\overline{p+qx(\alpha_{x}+p\alpha_{z})}dx_{0}-\overline{qx}dx_{1}+ \overline{1+qx\alpha_{z}}dx_{2}+\overline{x(1+q(\alpha_{p}+x\alpha_{z}))}dx_{3}\\
	\notag &\hspace{78mm}+\overline{y+qx(\alpha_{q}+y\alpha_{z})}dx_{4}.
	\end{flalign}
}
We next show that the function $\overline{y+\alpha x}$ is defined on $J^{1}(2,1)/\Ch(\calI)$.

\begin{lem}\label{lem:beta-function}
	The partial derivative $\overline{y+\alpha x}$ by $x_{0}$ is zero. Namely, $\overline{y + \alpha x}$ is a function defined on $U_{2} (\subset J^{1}(2,1)/\Ch(\calI))$.
\end{lem}

\begin{proof}
	Remark that 
	\[
		\overline{y+\alpha x} =(y+\alpha x) \circ \Phi^{-1}_{2}=\varphi_{2} + x_{0}x_{1},
	\]
	and the function $\alpha$ satisfies the following two partial differential equations:
	\begin{flalign*}
		\alpha_{x}+p\alpha_{z}-\alpha(\alpha_{y}+q\alpha_{z})&=0,\\
		1+x(\alpha_{y}+q\alpha_{z})&=0.
	\end{flalign*}
	Thus we obtain
	\begin{flalign}\label{formula-alpha-x}
		x(\alpha_{x}+p\alpha_{z})=-\alpha.
	\end{flalign}
	Therefore, by Lemma~\ref{lem:property-bar2}, (\ref{dphi2}) and (\ref{formula-alpha-x}), one has
	\begin{flalign*}
		(\overline{y+\alpha x})_{x_{0}}=(\varphi_{2})_{x_{0}} + x_{1} = \overline{x(\alpha_{x}+p\alpha_{z})} + \overline{\alpha}
		=-\overline{\alpha}+\overline{\alpha}=0,
	\end{flalign*}
	which completes the proof.
\end{proof}

\begin{prop}\label{prop:form-I/ChI-2}
	If the function $\alpha$ satisfies $1+x(\alpha_{y}+q\alpha_{z}) = 0$ locally, then the following holds:
	\[
		\calI=\{\pi^{\ast}(dx_{2}+(\overline{y+\alpha x}) dx_{4}), \pi^{\ast}(dx_{3}-x_{1}dx_{4})\}_{\mathrm{diff}},
	\]
	and thus we obtain
	\[
		\tilde{\calI}=\left\{dx_{2}+(\overline{y+\alpha x}) dx_{4},dx_{3}-x_{1}dx_{4}\right\}_{\mathrm{diff}}.
	\]	
\end{prop}

\begin{proof}
	In order to find the differential forms which generate $\tilde{\calI}$, we pull back $\omega_{0}$ and $\Psi$ to $V_{2} (\subset \R \times J^{1}(2,1)/\Ch(\calI))$. By (\ref{dphi2}) and (\ref{dpsi2}), we obtain
	\begin{flalign}
		\notag(\Phi_{2}^{-1})^{\ast}\omega_{0}
		\notag&=d\psi_{2} -x_{3}dx_{0}-x_{4}d\varphi_{2} \\
		\label{righhandside-defform2}&=dx_{2}+x_{0}dx_{3}+\overline{y}dx_{4},
	\end{flalign}
	and 
	\begin{flalign}\label{phi-inverse-form2}
		(\Phi_{2}^{-1})^{\ast}\Psi = dx_{3}-x_{1}dx_{4}.
	\end{flalign}
	Then, the right-hand side of (\ref{phi-inverse-form2}) is a 1-form on $U_{2}$. However, since the right-hand side of (\ref{righhandside-defform2}) has a fiber coordinate $x_{0}$, it is not a 1-form on $U_{2}$. To eliminate $x_{0}$, we perform the following calculation:
	\begin{flalign}
		\notag(\Phi_{2}^{-1})^{\ast}(\omega_{0}-x\Psi)&=dx_{2}+x_{0}dx_{3}+\overline{y}dx_{4}-x_{0}(dx_{3}-\overline{\alpha}dx_{4})\\
		\label{defform-omega0-xPsi2}&=dx_{2}+\overline{y+\alpha x}dx_{4}.
	\end{flalign}
	By Lemma~\ref{lem:beta-function}, the function $\overline{y+\alpha x}$ is on $U_{2}$. Therefore, the 1-form (\ref{defform-omega0-xPsi2}) is on $U_{2}$. 

From the above discussion, we find a candidate for a generator of $\tilde{\calI}$. We next show that these are the desired 1-forms. Then we have
	\begin{flalign*}
		\pi^{\ast}(dx_{2}+\overline{y+\alpha x}dx_{4})&=(\Phi_{2}^{\ast}\circ (\pr|_{V_{2}})^{\ast})(dx_{2}+\overline{y+\alpha x}dx_{4})=\Phi_{2}^{\ast} (\Phi_{2}^{-1})^{\ast}(\omega_{0}-x\Psi)\\
		&=\omega_{0}-x\Psi,\\
		\pi^{\ast}(dx_{3}-x_{1}dx_{4})&=(\Phi_{2}^{\ast}\circ (\pr|_{V_{2}})^{\ast})(dx_{3}-x_{1}dx_{4})=\Phi_{2}^{\ast}(\Phi_{2}^{-1})^{\ast}(\Psi)=\Psi.
	\end{flalign*}
	Therefore, one has 
	\[
		\calI=\{\pi^{\ast}(dx_{2}+\overline{y+\alpha x}dx_{4}),\pi^{\ast}(dx_{3}-x_{1}dx_{4})\}_{\mathrm{diff}}
	\]
	and thus 
	\[
		\tilde{\calI}=\{dx_{2}+\overline{y+\alpha x}dx_{4},dx_{3}-x_{1}dx_{4}\}_{\mathrm{diff}}.
	\]
	This completes the proof.
\end{proof}

\section{Derived system of the reduced GMAS}
In this section, we study derived systems of the reduced GMAS $\tilde{\calI}$. First of all, we define some basic notions from the theory of exterior differential systems. For the details, see \cite{BCG3,Cfb}.

\begin{Def}
	Let $M$ be a manifold, and $J$ be a subbundle of $T^{\ast}M$ with rank $k$. For a local basis $\alpha_{1},\dots,\alpha_{k}$ of $J$ defined on an open set $U$ of $M$, we define
	\[
		\mathcal{J}|_{U}:=\{\alpha_{1},\dots,\alpha_{k}\}_{\mathrm{diff}}.
	\]
	Then the differential ideal $\mathcal{J}$ is called a \textit{Pfaffian system} of rank $k$, and the subbundle $J$ is called the \textit{corresponding subbundle} of $\mathcal{J}$.
\end{Def}

\begin{rem}
	We recall that the GMAS $\calI$ is written by $\calI=\{\omega_{0},\Psi\}_{\mathrm{diff}}$, where
	\[
		\omega_{0}=dz-pdx-qdy, \hspace{5mm}\Psi \equiv dp-\alpha(x,y,z,p,q)dq \mod \omega_{0}.
	\]
	Since these are linearly independent 1-forms, $\calI$ is a Pfaffian system of rank 2.
\end{rem}

We next define derived systems of Pfaffian systems. Let $\mathcal{J}$ be a Pfaffian system on a manifold $M$, and $J$ be the corresponding subbundle of $\mathcal{J}$. Then, the following map
\[
	\delta\colon \Gamma(J) \to \Gamma\left(\bigwedge^{2} (T^{\ast}M/J)\right);\ \theta \mapsto d\theta \bmod J,
\]
induces the bundle map $\delta' \colon J \to \bigwedge^{2} (T^{\ast}M/J)$, since $\delta$ is a $C^{\infty}(M)$-linear map. Then, we put
\[
	J^{(1)}:=\Ker \delta'
\]
and $J^{(1)}$ is called the \textit{first derived system} of $J$. Inductively, if $J^{(k-1)}$ is a subbundle of $T^{\ast}M$, then we define $J^{(k)}:=(J^{(k-1)})^{(1)}$ and $J^{(k)}$ is called the \textit{$k$-th derived system}. 

\begin{rem}
	For derived systems of a Pfaffian system $J$, it is known that there exists $N \in \N$ such that 
\[
	J \supsetneq J^{(1)} \supsetneq J^{(2)} \supsetneq \cdots \supsetneq J^{(N)}=J^{(N+1)}=\cdots.
\]
	The number $N$ is called the \textit{derived length} of $J$.
\end{rem}

\begin{Def}
	Let $\mathcal{J}$ be a Pfaffian system on a manifold $M$, and $N$ be its derived length. We denote the corresponding subbundle of $\mathcal{J}$ by $J$. The pair
	\[
		(\corank J,\corank J^{(1)},\ldots,\corank J^{(N)})
	\] 
	is called the \textit{derived type} of $\mathcal{J}$. Then $\mathcal{J}$ is called of type $(\corank J,\corank J^{(1)},\ldots,$ $\corank J^{(N)})$.
\end{Def}

\subsection{Case of generic type}
In this subsection, for the GMAS $\calI$ which is of generic type, we study the derived systems of $\tilde{\calI}$. The following theorem is the main result in this subsection.
\begin{thm}\label{thm:criterion1}
	For the GMAS $\calI$ and $p_{0} \in J^{1}(2,1)$, the following holds:
	\begin{enumerate}[\normalfont (1)]
		\item $\tilde{\calI}$ is of type $(2,3)$ at $\pi(p_{0})$ if and only if $\alpha$ satisfies the following:
	\begin{flalign}\label{discriminant1-1}
			\left\{
				\begin{array}{l}
					\alpha_{p}(\alpha_{y}+q\alpha_{z})-\alpha(\alpha_{y}+q\alpha_{z})_{p}-(\alpha_{y}+q\alpha_{z})_{q}-\alpha_{z}=0 \\
					\alpha_{yy}+2q\alpha_{yz} +q^{2}\alpha_{zz}=0
				\end{array}
			\right. \hspace{4mm} \text{at $p_{0}$}.
		\end{flalign}
		\item $\tilde{\calI}$ is of type $(2,3,4)$ at $\pi(p_{0})$ if and only if $\alpha$ does not satisfy (\ref{discriminant1-1}) at $p_{0}$.
		\end{enumerate}
\end{thm}

In order to show Theorem~\ref{thm:criterion1}, we prepare some lemmas.

\begin{lem}\label{lem:criterion-distribution1}
	For the GMAS $\calI$ and $p_{0} \in J^{1}(2,1)$, the following holds:
		\begin{enumerate}[\normalfont (1)]
			\item $\tilde{\calI}$ is of type $(2,3)$ at $\pi(p_{0})$ if and only if $\overline{\alpha}$ satisfies	
			\begin{flalign}\label{discriminant1-2}
			\left\{
				\begin{array}{l}
					x_{1}(\overline{\alpha})_{x_{1}x_{2}}-\overline{\alpha}(\overline{\alpha})_{x_{1}x_{3}}-(\overline{\alpha})_{x_{1}x_{4}}+(\overline{\alpha})_{x_{1}}(\overline{\alpha})_{x_{3}}-(\overline{\alpha})_{x_{2}}=0\\
					(\overline{\alpha})_{x_{1}x_{1}}=0,
				\end{array}
			\right. \hspace{4mm} \text{at $p_{0}$}.
		\end{flalign}
		\item $\tilde{\calI}$ is of type $(2,3)$ at $\pi(p_{0})$ if and only if $\overline{\alpha}$ does not satisfy (\ref{discriminant1-2}) at $p_{0}$. 
		\end{enumerate}	
\end{lem}

\begin{proof}
	By Proposition~\ref{prop:form-I/ChI-1}, we have
	\[
		\tilde{\calI}=\{dx_{2}+x_{1}dx_{4},dx_{3}-\overline{\alpha}dx_{4}\}_{\mathrm{diff}},
	\]
	and thus the corresponding subbundle $\tilde{I}$ of $\tilde{\calI}$ is obtained by
	\[
		\tilde{I}=\Span\{dx_{2}+x_{1}dx_{4},dx_{3}-\overline{\alpha}dx_{4}\}.
	\]
	Then we put
	\[
		\xi_{1}:=dx_{2}+x_{1}dx_{4},\hspace{3mm}\xi_{2}:=dx_{3}-\overline{\alpha}dx_{4},
	\]
	and take a local coframe $\{\xi_{1},\xi_{2},dx_{1},dx_{4}\}$ of $T^{\ast}U_{1}$. We obtain
	\begin{flalign*}
		d\xi_{1}&=dx_{1}\wedge dx_{4},  \\
		d\xi_{2} &=-d\overline{\alpha} \wedge dx_{3}
		=-((\overline{\alpha})_{x_{1}}dx_{1}+(\overline{\alpha})_{x_{2}}dx_{2}+(\overline{\alpha})_{x_{3}}dx_{3}+(\overline{\alpha})_{x_{4}}dx_{4}) \wedge dx_{4}\\
		&\equiv -(\overline{\alpha})_{x_{1}} dx_{1}\wedge dx_{4} \mod \xi_{1},\xi_{2}.
	\end{flalign*}
	Therefore we put
	\begin{flalign*}
		\eta_{1}:=\xi_{1},\hspace{3mm}
		\eta_{2}:=\xi_{2}+(\overline{\alpha})_{x_{1}} \xi_{1}=dx_{3}+(\overline{\alpha})_{x_{1}}dx_{2}+(x_{1}(\overline{\alpha})_{x_{1}}-\overline{\alpha})dx_{4},
	\end{flalign*}
	and change the defining 1-forms $\xi_{1}$ and $\xi_{2}$ of $\tilde{I}$ to $\eta_{1}$ and $\eta_{2}$. Then we have
	\begin{flalign*}
		d\eta_{1} &=dx_{1}\wedge dx_{4} \not \equiv 0 \mod \eta_{1},\eta_{2}, \\
		d\eta_{2} &\equiv d\xi_{2}+ (\overline{\alpha})_{x_{1}} d\xi_{1} \equiv 0 \mod \eta_{1},\eta_{2},
	\end{flalign*}
	and thus, one has $(\tilde{I})^{(1)}=\Span \{\eta_{2}\}$. Moreover, we have
	\begin{flalign*}
		d\eta_{2} &= d(\overline{\alpha})_{x_{1}}\wedge dx_{2} + d(x_{1}(\overline{\alpha})_{x_{1}}-\overline{\alpha}) \wedge dx_{4} \\
		&=((\overline{\alpha})_{x_{1}x_{1}}dx_{1} + (\overline{\alpha})_{x_{1}x_{3}}dx_{3}+(\overline{\alpha})_{x_{1}x_{4}}dx_{4}) \wedge dx_{2} \\
		&\hspace{20mm}+((\overline{\alpha})_{x_{1}x_{1}}dx_{1} + (x_{1}(\overline{\alpha})_{x_{1}x_{2}} -(\overline{\alpha})_{x_{2}})dx_{2}\\
		&\hspace{40mm}+ (x_{1}(\overline{\alpha})_{x_{1}x_{3}} -(\overline{\alpha})_{x_{3}})dx_{3})\wedge dx_{4} \\
		&\equiv ((\overline{\alpha})_{x_{1}x_{1}}dx_{1}+(\overline{\alpha})_{x_{1}x_{3}}(\overline{\alpha}-x_{1}(\overline{\alpha})_{x_{1}})dx_{4} + (\overline{\alpha})_{x_{1}x_{4}}dx_{4})\wedge dx_{2} \\
		&\hspace{20mm}+(x_{1}(\overline{\alpha})_{x_{1}x_{1}}dx_{1} +(x_{1}(\overline{\alpha})_{x_{1}x_{2}}-(\overline{\alpha})_{x_{2}})dx_{2} \\
		&\hspace{40mm}-(\overline{\alpha})_{x_{1}}(x_{1}(\overline{\alpha})_{x_{1}x_{3}} - (\overline{\alpha})_{x_{3}})dx_{2}) \wedge dx_{4} \mod \eta_{2} \\
		&\equiv (x_{1}(\overline{\alpha})_{x_{1}x_{2}}-\overline{\alpha}(\overline{\alpha})_{x_{1}x_{3}}-(\overline{\alpha})_{x_{1}x_{4}}+(\overline{\alpha})_{x_{1}}(\overline{\alpha})_{x_{3}}-(\overline{\alpha})_{x_{2}}) \eta_{1} \wedge dx_{4}\\
		&\hspace{70mm}+ (\overline{\alpha})_{x_{1}x_{1}} dx_{1} \wedge \eta_{1} \mod \eta_{2}.
	\end{flalign*}
	Therefore, if $\overline{\alpha}$ satisfies (\ref{discriminant1-2}), then we have
	\[
		d\eta_{2} \equiv 0 \mod \eta_{2},
	\]
	and thus $(\tilde{I})^{(2)}=(\tilde{I})^{(1)}$, otherwise $(\tilde{I})^{(2)}=\{0\}$. This completes the proof.
\end{proof}

In order to express the condition (\ref{discriminant1-2}) by using functions defined on $J^{1}(2,1)$, we next calculate the partial derivatives of $\overline{\alpha}$.

\begin{lem}\label{lem:relation-alpha1}
	For the function $\overline{\alpha}$, the following holds:
	\begin{enumerate}[\normalfont (1)]
		\item $(\overline{\alpha})_{x_{1}}(\overline{\alpha})_{x_{3}}-(\overline{\alpha})_{x_{2}}=\overline{\rho^{2}(\alpha_{p}(\alpha_{y}+q \alpha_{z})-\alpha_{z})}$.
		\item $(\overline{\alpha})_{x_{1}x_{1}}=\overline{\rho^{2}(\alpha_{yy}+2q\alpha_{yz}+q^{2}\alpha_{zz})}$.
		\item $(\overline{\alpha})_{x_{1}x_{2}}=\overline{\rho^{2}(\alpha_{yz}+q\alpha_{zz}-\rho x \alpha_{z}(\alpha_{yy}+2q\alpha_{yz}+q^{2}\alpha_{zz}))} $ .
		\item $(\overline{\alpha})_{x_{1}x_{3}}=\overline{\rho^{2}(\alpha_{yp}+q \alpha_{zp}+x(\alpha_{yz}+q \alpha_{zz})}$
		
		\hspace{30mm}$\overline{-\rho x (\alpha_{p}+x\alpha_{z})(\alpha_{yy}+2q\alpha_{yz}+q^{2}\alpha_{zz}))}$.
		\item $(\overline{\alpha})_{x_{1}x_{4}}=\overline{\rho^{2}(y(\alpha_{yz}+q\alpha_{zz})+\alpha_{yq}+q\alpha_{zq}+\alpha_{z}}$
		
		\hspace{30mm}$\overline{-\rho x (\alpha_{q}+y\alpha_{z})(\alpha_{yy}+2q\alpha_{yz}+q^{2}\alpha_{zz}))}$.
	\end{enumerate}
\end{lem}

\begin{proof}
	We have only to show (1), since one can show other cases, similarly. We remark that $\overline{\alpha}(x_{1},x_{2},x_{3},x_{4}):=\alpha(x,\varphi_{1},\psi_{1},x_{3},x_{4})$. By (\ref{dphi1}), (\ref{dpsi1}) and the chain rule, we have
	\begin{flalign*}
		(\overline{\alpha})_{x_{1}}&=\overline{\alpha_{y}}\varphi_{x_{1}} + \overline{\alpha_{z}}\psi_{x_{1}}
		=\overline{\rho(\alpha_{y}+q\alpha_{z})},\\
		(\overline{\alpha})_{x_{2}} &= \overline{\alpha_{y}}\varphi_{x_{2}} + \overline{\alpha_{z}}\psi_{x_{2}}=\overline{\alpha_{z}(1-\rho x (\alpha_{y}+q\alpha_{z}))}=\overline{\rho \alpha_{z}},\\
		(\overline{\alpha})_{x_{3}}&=\overline{\alpha_{y}}\varphi_{x_{3}} + \overline{\alpha_{z}}\psi_{x_{3}} + \overline{\alpha_{p}}=\overline{(\alpha_{p}+x\alpha_{z})(1-\rho x (\alpha_{y}+q\alpha_{z}))}=\overline{\rho(\alpha_{p}+x\alpha_{z})}.
	\end{flalign*}
	Therefore, one has
	\begin{flalign*}
		(\overline{\alpha})_{x_{1}}(\overline{\alpha})_{x_{3}}-(\overline{\alpha})_{x_{2}}=\overline{\rho^{2}(\alpha_{y}+q\alpha_{z})(\alpha_{p}+x\alpha_{z})}-\overline{\rho\alpha_{z}}=\overline{\rho^{2}(\alpha_{p}(\alpha_{y}+q \alpha_{z})-\alpha_{z})},
	\end{flalign*}
	which completes the proof of (1).
\end{proof}

\begin{lem}\label{lem:criterion-condition1}
	For a point $p_{0} \in J^{1}(2,1)$, the function $\overline{\alpha}$ satisfies 
		\begin{flalign}\label{alpha-bar-condition1}
			\left\{
				\begin{array}{l}
					x_{1}(\overline{\alpha})_{x_{1}x_{2}}-\overline{\alpha}(\overline{\alpha})_{x_{1}x_{3}}-(\overline{\alpha})_{x_{1}x_{4}}+(\overline{\alpha})_{x_{1}}(\overline{\alpha})_{x_{3}}-(\overline{\alpha})_{x_{2}}=0\\
					(\overline{\alpha})_{x_{1}x_{1}}=0,
				\end{array}
			\right.\hspace{4mm} \text{at $p_{0}$},
		\end{flalign}
		if and only if, the function $\alpha$ satisfies
		\begin{flalign}\label{alpha-condition1}
			\left\{
				\begin{array}{l}
					\alpha_{p}(\alpha_{y}+q\alpha_{z})-\alpha(\alpha_{y}+q\alpha_{z})_{p}-(\alpha_{y}+q\alpha_{z})_{q}-\alpha_{z}=0 \\
					\alpha_{yy}+2q\alpha_{yz} +q^{2}\alpha_{zz}=0,
				\end{array}
			\right.\hspace{4mm}\text{at $p_{0}$}.
		\end{flalign}
\end{lem}

\begin{proof}
	We assume that $\overline{\alpha}$ satisfies (\ref{alpha-bar-condition1}). Then, by Lemma~\ref{lem:relation-alpha1} (1), we have
	\[
		(\overline{\alpha})_{x_{1}x_{1}}=\overline{\rho^{2}(\alpha_{yy}+2q\alpha_{yz}+q^{2}\alpha_{zz})}.
	\]
	Since $\overline{\alpha}$ satisfies $(\overline{\alpha})_{x_{1}x_{1}}=0$ and $\rho$ does not take the value 0, we obtain
	\begin{flalign}\label{condition-alpha-yy}
		\alpha_{yy}+2q\alpha_{yz}+q^{2}\alpha_{zz}=0.
	\end{flalign}
	Then, by Lemma~\ref{lem:relation-alpha1} and (\ref{condition-alpha-yy}), we have
	\begin{flalign*}
		0&=x_{1}(\overline{\alpha})_{x_{1}x_{2}}-\overline{\alpha}(\overline{\alpha})_{x_{1}x_{3}}-(\overline{\alpha})_{x_{1}x_{4}}+(\overline{\alpha})_{x_{1}}(\overline{\alpha})_{x_{3}}-(\overline{\alpha})_{x_{2}}\\
		&=\overline{\rho^{2} (y+\alpha x)(\alpha_{yz}+q\alpha_{zz})}-\overline{\rho^{2}\alpha(x(\alpha_{yz}+q \alpha_{zz})+\alpha_{yp}+q \alpha_{zp})}\\
		&\hspace{10mm}-\overline{\rho^{2}(y(\alpha_{yz}+q\alpha_{zz})+\alpha_{yq}+q\alpha_{zq}+\alpha_{z})}+\overline{\rho^{2}(\alpha_{p}(\alpha_{y}+q \alpha_{z})-\alpha_{z})} \\
		&=\overline{\rho^{2}(\alpha_{p}(\alpha_{y}+q\alpha_{z})-\alpha(\alpha_{yp}+q\alpha_{zp})-(\alpha_{yq}+q\alpha_{zq})-2\alpha_{z})}\\
		&=\overline{\rho^{2}(\alpha_{p}(\alpha_{y}+q\alpha_{z})-\alpha(\alpha_{y}+q\alpha_{z})_{p}-(\alpha_{y}+q\alpha_{z})_{q}-\alpha_{z})}.
	\end{flalign*}
	Therefore, since again $\rho$ does not take the value 0, the function $\alpha$ satisfies 
	\[
		\alpha_{p}(\alpha_{y}+q\alpha_{z})-\alpha(\alpha_{y}+q\alpha_{z})_{p}-(\alpha_{y}+q\alpha_{z})_{q}-\alpha_{z}=0.
	\]
	On the other hand, the converse can be proved by direct calculations. This completes the proof. 
\end{proof}

\begin{proof}[Proof of Theorem~\ref{thm:criterion1}]
	This theorem is proved by Lemma~\ref{lem:criterion-distribution1} and Lemma~\ref{lem:criterion-condition1}, immediately.
\end{proof}

\begin{ex}\label{ex:structure-1}
	We consider $\alpha(x,y,z,p,q):=f(z-px-qy,p,q)$, where $f$ is any function with $u$, $v$ and $w$ as variables. By Example~\ref{ex:involutive-function}, $\alpha$ is of involutive type. Moreover, we have
	\[
		1+x(\alpha_{y}+q\alpha_{z})=1+x(-qf_{u}+qf_{u})=1 \ne 0.
	\]
	Therefore, $\alpha$ is of generic type. Furthermore, by Proposition~\ref{prop:form-I/ChI-1}, we obtain
	\[
		\tilde{\calI}=\{dx_{2}+x_{1}dx_{4},dx_{3}-\overline{f}dx_{4}\}_{\mathrm{diff}}.
	\]
	Furthermore, one has
	\begin{flalign*}
		&\hspace{5mm}\alpha_{p}(\alpha_{y}+q\alpha_{z})-\alpha(\alpha_{y}+q\alpha_{z})_{p}-(\alpha_{y}+q\alpha_{z})_{q}-\alpha_{z}\\
		&=f_{p}(-qf_{u}+qf_{u})-f(-qf_{u}+qf_{u})_{p}-(-qf_{u}+qf{u})_{q}-f_{u}\\
		&=-f_{u},
	\end{flalign*}
	and thus
	\begin{flalign*}
		\alpha_{yy}+2q\alpha_{yz} +q^{2}\alpha_{zz} = q^{2}f_{uu}-2q^{2} f_{uu}+q^{2}f_{uu}=0.
	\end{flalign*}
	Hence, by Theorem~\ref{thm:criterion1}, if the partial derivative $f$ by $u$ is zero, then $\tilde{\calI}$ is of type $(2,3)$. Conversely, if the partial derivative $f$ by $u$ is not zero, then $\tilde{\calI}$ is of type $(2,3,4)$.
\end{ex}

\subsection{Case of non-generic type}
In this subsection, for the GMAS $\calI$ which is of non-generic type, we study the derived systems of $\tilde{\calI}$. Hereafter, we put
\[
	\beta:=y+\alpha x.
\]
The following theorem is the main result in this subsection.
\begin{thm}\label{thm:criterion2}
	For the GMAS $\calI$ and $p_{0}\in J^{1}(2,1)$, the following holds:
	\begin{enumerate}[\normalfont (1)]
		\item $\tilde{\calI}$ is of type $(2,3)$ at $\pi(p_{0})$ if and only if $\alpha$ satisfies
		\begin{flalign}\label{discriminant2}
			\alpha_{p} + x \alpha_{z}=0, \hspace{4mm} \text{at $p_{0}$}.
		\end{flalign}
		\item $\tilde{\calI}$ is of type $(2,3,4)$ at $\pi(p_{0})$ if and only if $\alpha$ does not satisfy (\ref{discriminant2}) at $p_{0}$.
	\end{enumerate}
		\end{thm}
		
In order to show Theorem~\ref{thm:criterion2}, we prepare some lemmas.

\begin{lem}\label{lem:criterion-distribution2}
	For the GMAS $\calI$ and $p_{0}\in J^{1}(2,1)$, the following holds:
	\begin{enumerate}[\normalfont (1)]
		\item $\tilde{\calI}$ is of type $(2,3)$ at $\pi(p_{0})$ if and only if $\overline{\beta}$ satisfies
		\begin{flalign}\label{discriminant2-1}
			\left\{
				\begin{array}{l}
					\overline{\beta} (\overline{\beta})_{x_{1}x_{2}}-x_{1}(\overline{\beta})_{x_{1}x_{3}}-(\overline{\beta})_{x_{1}x_{4}}-(\overline{\beta})_{x_{1}}(\overline{\beta})_{x_{2}}+(\overline{\beta})_{x_{3}}=0\\
					(\overline{\beta})_{x_{1}x_{1}}=0,
				\end{array}
			\right.\hspace{4mm} \text{at $p_{0}$}.
		\end{flalign}
		\item $\tilde{\calI}$ is of type $(2,3,4)$ at $\pi(p_{0})$ if and only if $\overline{\beta}$ does not satisfy (\ref{discriminant2-1}) at $p_{0}$.
	\end{enumerate}
		\end{lem}

\begin{proof}
	By Proposition~\ref{prop:form-I/ChI-2}, we have
	\[
		\tilde{\calI}=\{dx_{2}+\overline{\beta}dx_{4},dx_{3}-x_{1}dx_{4}\}_{\mathrm{diff}},
	\]
	and thus the corresponding subbundle $\tilde{I}$ of $\calI$ is obtained by 
	\[
		\tilde{I}=\Span\{dx_{2}+\overline{\beta}dx_{4},dx_{3}-x_{1}dx_{4}\}.
	\]
	Then we put 
	\[
		\xi_{1}:=dx_{2}+\overline{\beta}dx_{4},\hspace{3mm}\xi_{2}:=dx_{3}-x_{1}dx_{4},
	\]
	and take a local coframe $\{\xi_{1},\xi_{2},dx_{1},dx_{4}\}$ of $T^{\ast}U_{2}$. Then we obtain
\begin{flalign*}
		d\xi_{1}&=d\overline{\beta} \wedge dx_{4}=((\overline{\beta})_{x_{1}}dx_{1}+(\overline{\beta})_{x_{2}}dx_{2}+(\overline{\beta})_{x_{3}}dx_{3}) \wedge dx_{4} \\
		&\equiv (\overline{\beta})_{x_{1}} dx_{1} \wedge dx_{4}\mod \xi_{1},\xi_{2},\\
		d\xi_{2} &=-dx_{1} \wedge dx_{4}.
	\end{flalign*}
	Therefore we put
	\begin{flalign*}
		\eta_{1}&:=\xi_{1}+(\overline{\beta})_{x_{1}}\xi_{2}=dx_{2}+(\overline{\beta})_{x_{1}}dx_{3} + (\overline{\beta}-x_{1}(\overline{\beta})_{x_{1}})dx_{4},\\
		\eta_{2}&:=\xi_{2}.
	\end{flalign*}
	and change the defining 1-forms $\xi_{1}$ and $\xi_{2}$ of $\tilde{I}$ to $\eta_{1}$ and $\eta_{2}$. Then we have
	\begin{flalign*}
		d\eta_{1} &\equiv d\xi_{1} + (\overline{\beta})_{x_{1}} d\xi_{2} \equiv 0 \mod \eta_{1},\eta_{2}, \\
		d\eta_{2} &=-dx_{1} \wedge dx_{4} \not \equiv 0 \mod \eta_{1},\eta_{2},
	\end{flalign*}
	and thus, one has $(\tilde{I})^{(1)}=\Span\{\eta_{2}\}$. Moreover, we have
	\begin{flalign*}
		d\eta_{1} &= d(\overline{\beta})_{x_{1}} \wedge dx_{3} + d(\overline{\beta}-x_{1}(\overline{\beta})_{x_{1}})\wedge dx_{4} \\
		&= ((\overline{\beta})_{x_{1}x_{1}}dx_{1}+(\overline{\beta})_{x_{1}x_{2}}dx_{2}+(\overline{\beta})_{x_{1}x_{4}}dx_{4}) \wedge dx_{3} \\
		&\hspace{20mm}+(-x_{1}(\overline{\beta})_{x_{1}x_{1}}dx_{1} + ((\overline{\beta})_{x_{2}} -x_{1}(\overline{\beta})_{x_{1}x_{2}})dx_{2}\\
		&\hspace{50mm}+((\overline{\beta})_{x_{3}} - x_{1}(\overline{\beta})_{x_{1}x_{3}})dx_{3}) \wedge dx_{4}\\
		&\equiv ((\overline{\beta})_{x_{1}x_{1}}dx_{1}+(\overline{\beta})_{x_{1}x_{2}} (x_{1}(\overline{\beta})_{x_{1}} -\overline{\beta})dx_{2}+(\overline{\beta})_{x_{1}x_{4}}dx_{4}) \wedge dx_{3} \\
		&\hspace{20mm}+(-x_{1}(\overline{\beta})_{x_{1}x_{1}}dx_{1} -(\overline{\beta})_{x_{1}} ((\overline{\beta})_{x_{2}} -x_{1}(\overline{\beta})_{x_{1}x_{2}})dx_{3} \\ 
		&\hspace{50mm}+((\overline{\beta})_{x_{3}} - x_{1}(\overline{\beta})_{x_{1}x_{3}})dx_{3}) \wedge dx_{4} \mod \eta_{1}\\
		&\equiv(\overline{\beta} (\overline{\beta})_{x_{1}x_{2}}-x_{1}(\overline{\beta})_{x_{1}x_{3}}-(\overline{\beta})_{x_{1}x_{4}}-(\overline{\beta})_{x_{1}}(\overline{\beta})_{x_{2}}+(\overline{\beta})_{x_{3}})\eta_{2} \wedge dx_{4}\\
		&\hspace{71.5mm}+ (\overline{\beta})_{x_{1}x_{1}}dx_{1} \wedge \eta_{2} \mod \eta_{1}.
	\end{flalign*}
	Therefore, if $\overline{\beta}$ satisfies (\ref{discriminant1-2}), then we have
	\[
		d\eta_{1} \equiv 0 \mod \eta_{1},
	\]
	and thus $(\tilde{I})^{(2)}=(\tilde{I})^{(1)}$, otherwise $(\tilde{I})^{(2)}=\{0\}$. This completes the proof.
\end{proof} 

In order to express the condition (\ref{discriminant2-1}) by using functions defined on $J^{1}(2,1)$, we next calculate the partial derivatives of $\overline{\beta}$.

\begin{lem}\label{lem:relation-beta}
	For the function $\overline{\beta}$, the following holds:
	\[
		(\overline{\beta})_{x_{1}}=0,\hspace{2mm}
		(\overline{\beta})_{x_{2}}=\overline{x\alpha_{z}},\hspace{2mm} 
		(\overline{\beta})_{x_{3}}=\overline{x(\alpha_{p}+x\alpha_{z})}.
		\]
\end{lem}

\begin{proof}
	By (\ref{dphi2}), (\ref{dpsi2}) and the chain rule, we have
	\begin{flalign*}
		(\overline{\beta})_{x_{1}}&=(\varphi_{2})_{x_{1}}+x_{0}=-x_{0}+x_{0}=0,\\
		(\overline{\beta})_{x_{2}}&=(\varphi_{2})_{x_{2}}=\overline{x\alpha_{z}},\\
		(\overline{\beta})_{x_{3}}&=(\varphi_{2})_{x_{3}}=\overline{x(\alpha_{p}+x\alpha_{z})},
	\end{flalign*}
	which completes the proof.
\end{proof}

\begin{lem}\label{lem:criterion-condition2}
	For a point $p_{0} \in J^{1}(2,1)$, the function $\overline{\beta}$ satisfies
	\begin{flalign}\label{beta-bar-condition}
			\left\{
				\begin{array}{l}
					\overline{\beta} (\overline{\beta})_{x_{1}x_{2}}-x_{1}(\overline{\beta})_{x_{1}x_{3}}-(\overline{\beta})_{x_{1}x_{4}}-(\overline{\beta})_{x_{1}}(\overline{\beta})_{x_{2}}+(\overline{\beta})_{x_{3}}=0\\
					(\overline{\beta})_{x_{1}x_{1}}=0,
				\end{array}
			\right. \hspace{4mm} \text{at $p_{0}$}.
		\end{flalign}
		if and only if, the function $\alpha$ satisfies
		\begin{flalign}\label{alpha-condition1}
				\begin{array}{l}
					\alpha_{p} + x \alpha_{z}=0,\hspace{4mm}\text{at $p_{0}$}.
				\end{array}
		\end{flalign}
\end{lem}

\begin{proof}
	First of all, we assume that $\overline{\beta}$ satisfies (\ref{beta-bar-condition}). By Lemma~\ref{lem:relation-beta}(1), we have $(\overline{\beta})_{x_{1}}=0$. Then, by the first formula of (\ref{beta-bar-condition}) and Lemma~\ref{lem:relation-beta} (3), we obtain
	\begin{flalign*}
		0=\overline{\beta} (\overline{\beta})_{x_{1}x_{2}}-x_{1}(\overline{\beta})_{x_{1}x_{3}}-(\overline{\beta})_{x_{1}x_{4}}-(\overline{\beta})_{x_{1}}(\overline{\beta})_{x_{2}}+(\overline{\beta})_{x_{3}}=(\overline{\beta})_{x_{3}}=\overline{x(\alpha_{p}+x\alpha_{z})}.
	\end{flalign*}
	In this subsection, we assume that $\alpha$ satisfies
	\[
		1+x(\alpha_{y}+x\alpha_{z})=0.
	\]
	Hence, it must be $x \ne 0$. Therefore, $\alpha$ satisfies
	\[
		\alpha_{p}+x\alpha_{z}=0.
	\]
	
	On the other hand, the converse can be proved by direct calculations. This completes the proof.
	\end{proof}

\begin{proof}[Proof of Theorem~\ref{thm:criterion2}]
	This theorem is proved by Lemma~\ref{lem:criterion-distribution2} and Lemma~\ref{lem:criterion-condition2}, immediately.
\end{proof}

\begin{ex}\label{ex:structure-2}
	\begin{enumerate}[\normalfont (1)]
		\item We consider the case of $\alpha:=(g(p,q)-y)/x$, where $g(p,q)$ is an arbitrary function. By Example~\ref{ex:involutive-function}, $\alpha$ is of involutive type. Moreover we obtain
	\[
		1+x(\alpha_{y}+q\alpha_{z})=1+x(-1/x)=0,
	\]
	Therefore, $\alpha$ is of non-generic type. Furthermore, one has
	\begin{flalign*}
		\alpha_{p}+x\alpha_{z} = g_{p}(p,q)/x.	
	\end{flalign*}
	Hence, by using Theorem~\ref{thm:criterion2}, if $g_{p}=0$, then $\tilde{\calI}$ is of type (2,3). Otherwise $\tilde{\calI}$ is of type (2,3,4).
	\item Let $C$ be a constant, and we consider the case of $\alpha:=-z/(qx)+C/x+p/q$. By Example~\ref{ex:involutive-function}, $\alpha$ is also of involutive type. Moreover we obtain
	\[
		1+x(\alpha_{y}+q\alpha_{z})=1+x(-q/(qx))=0,
	\]
	Therefore $\alpha$ is of non-generic type. Then we have
	\begin{flalign*}
		\alpha_{p}+x\alpha_{z}=1/q+x (-1/(qx))=0
	\end{flalign*}
	Hence, by using Theorem~\ref{thm:criterion2}, $\tilde{\calI}$ is of type (2,3).
	\end{enumerate}
\end{ex}

\section{Constructions of geometric singular solutions}
In this section, we construct geometric singular solutions of the GMAE (\ref{ODS1}). Moreover, we give criteria for the geometric singular solution to be right-left equivalent to the cuspidal edge, swallowtail and butterfly.

First of all, we define an integral manifold of an EDS.\begin{Def}
	Let $\mathcal{J}$ be an exterior differential system on a manifold $M$. A submanifold $\iota \colon N \hookrightarrow M$ is called an \textit{integral manifold} of $\mathcal{J}$, if $N$ satisfies $\iota^{\ast}\mathcal{J}=\{0\}$.
\end{Def}

We next define a geometric singular solution of the GMAE(\ref{ODS1}).
\begin{Def}\label{def:singular-sol}
	Let $\calI$ be the GMAS corresponding to the GMAE (\ref{ODS1}). A two-dimensional integral manifold $S$ of $\calI$ is called a \textit{geometric singular solution} of (\ref{ODS1}) if $\varpi|_{S} \colon S \to J^{0}(2,1)$ is not an immersion, where $\varpi \colon J^{1}(2,1) \to J^{0}(2,1)$ is the natural projection.  Hereafter, if there is no danger of confusion, we shall simply denote $\varpi|_{S}$ by $\varpi$.
\end{Def}

\begin{rem}
	Let $S$ be a two-dimensional integral manifold of $\calI$. If $(dx \wedge dy)|_{S} \ne 0$, then $S$ is called a classical solution of the GMAE (\ref{ODS1}).
\end{rem}

\begin{rem}
	We remark that there is an example of a geometric singular solution which the measure of the set of all singular points is not zero. The following is an example of such a geometric singular solution: let us consider $\alpha=0$ case, that is
	\begin{flalign}\label{GMAE-zero}
		\left\{
			\begin{array}{l}	
				z_{xx}=0 \\
				z_{xy}=0.
			\end{array}
		\right.
	\end{flalign}
	By direct calculations, the set
	\[
		S:=\{(s,0,1,0,t)\ |\ s\in I,\ t\in J\}
	\]
	is a two-dimensional integral manifold of the GMAS corresponding to (\ref{GMAE-zero}), where $I$ and $J$ are some open sets in $\R$. Moreover, the set of all singular points of $\varpi \colon S \to J^{0}(2,1)$ is $I \times J$. Therefore, $S$ is a geometric singular solution where the measure of the set of all singular points is not zero. In this section, we consider geometric singular solutions where the measure of the set of all singular points is zero. 
\end{rem}

In order to study the singularities of geometric singular solutions, we consider the singularities of the map $\varpi \colon S \to J^{0}(2,1)$. Namely, since $S$ is  two-dimensional and $J^{0}(2,1)$ is three-dimensional, we locally study the singularities of the map from $\R^{2}$ to $\R^{3}$. 

\subsection{Preliminaries from singularity theory}
In this subsection, we explain some well-known facts about singularity theory. Especially, we describe criteria for singularities of a map from $\R^{2}$ to $\R^{3}$. For the details of this subsection, see \cite{USY}.

\begin{Def}
	Let $M$ be an $m$-dimensional manifold, and $N$ be an $(m+1)$-dimensional manifold. A map $f \colon M \to N$ is called a \textit{wave front} if there exists Legendre immersion $L \colon M \to P(T^{\ast}N)$ such that $f = \pi \circ L$ holds. Here the map $\pi \colon P(T^{\ast}N) \to N$ is the natural projection. 
\end{Def}

\begin{Def}
	Let $E$ be a contact manifold, and $\varpi \colon E \to N$ be a fiber bundle. If any fibers are Legendrian submanifold in $E$, then $\varpi \colon E \to N$ is called a \textit{Legendrian fibration}.
\end{Def}

\begin{ex}
	The natural projection $\varpi \colon J^{1}(2,1) \to J^{0}(2,1)$ is a Legendrian fibration. 
\end{ex}

We next introduce a well-known theorem which describes the relationship between Legendre fibrations and wave fronts

\begin{thm}
	The projection of a Legendre submanifold by a Legendrian fibration is a wave front.
\end{thm}

\begin{Def}
	Two map germs $f\colon (\R^{2},a) \to (\R^{3},f(a))$ and $g\colon (\R^{2},b) \to (\R^{3},g(b))$ are \textit{right-left equivalent}, if there exist two diffeomorphism-germs $\varphi \colon (\R^{2},a) \to (\R^{2},b)$ and $\psi \colon (\R^{3},f(a)) \to (\R^{2},g(a))$ such that the following diagram commute:
	\begin{center}
		\begin{tikzpicture}[auto]
			\node (a) at (0, 0) {$(\R^{2},a)$}; 
			\node (b) at (0, -1.5) {$(\R^{2},b)$};
			\node (c) at (2, 0) {$(\R^{3},f(a))$};
			\node (d) at (2, -1.5) {$(\R^{3},g(b))$};
			\draw[->] (a) to node[swap] {$\scriptstyle \varphi$} (b);
			\draw[->] (c) to node {$\scriptstyle \psi$} (d);
			\draw[->] (a) to node {$\scriptstyle f$} (c);
			\draw[->] (b) to node[swap] {$\scriptstyle g$} (d);
		\end{tikzpicture}
		\end{center}
\end{Def}

From now on in this subsection, let $U$ be an open neighborhood of $p \in \R^{2}$, and $f \colon U \to \R^{3}$ be a wave front. We denote $S(f)$ by the set of all singular points of $f$. Then let $\tilde{\nu} \colon U \to S^{2}$ be a unit normal vector field of $f$. Now we define the map $\lambda \colon U \to \R$ by
\[
	\lambda(u,v):=\det (f_{u},f_{v},\tilde{\nu}),
\]
where $(u,v)$ is a local coordinate system of $U$. For the function $\lambda$, it is known that there exist a function $\hat{\lambda}$ and $\mu$ defined on $U$ such that $\lambda = \hat{\lambda}\mu$ and $\hat{\lambda}^{-1}(0) = S(f)$. The function $\hat{\lambda}$ is called the \textit{singularity identifier} of $f$.

\begin{Def}
	A point $p \in S(f)$ is called \textit{non-degenerate} if $d\hat{\lambda}_{p} \ne 0$, and \textit{degenerate} if $d\hat{\lambda}_{p} = 0$.
\end{Def}

We take a singular point $p \in U$ of $f$, and assume that $p$ is non-degenerate. Then, by the implicit function theorem, there exist an open neighborhood $U'$ of $p$ and a regular curve $\gamma \colon (-\varepsilon, \varepsilon) \to U'$ with $\gamma(0)=p$ such that $\hat{\lambda}(\gamma(t))=0$ holds for all $t \in U'$. This curve $\gamma$ is called a \textit{singular curve}. Moreover, we call $\hat{\gamma}:=f \circ \gamma$ by a \textit{singular locus}. Furthermore, since a non-degenerate singular point $p$ satisfies $\rank df_{p} =1$, there exists non-zero vector field $\eta$ on $U$ such that $df_{q}(\eta_{q}) =0$ for all $q \in S(f) \cap U$. This vector field $\eta$ is called a \textit{null vector field} along the singular curve $\gamma$.

\begin{Def}
	Let $f \colon (U,p) \to (\R^{3},f(p))$ be a map germ.
	\begin{enumerate}[(1)]
		\item If $f$ is right-left equivalent to the map germ $(u,v) \mapsto (u,v^{2},v^{3})$, then $f$ is called a \textit{cuspidal edge} at $p$.
		\item  If $f$ is right-left equivalent to the map germ $(u,v) \mapsto (u,3v^{4}+uv^{2},4v^{3}+2uv)$, then $f$ is called a \textit{swallowtail} at $p$.
		\item If $f$ is right-left equivalent to the map germ $(u,v) \mapsto (u,5v^{4}+2uv,uv^{2}+4v^{5})$, then $f$ is called a \textit{butterfly} at $p$.
		\item If $f$ is right-left equivalent to the map germ $(u,v) \mapsto (u,v(u^{2}-v^{2}),v^{2}(2u^{2}-3v^{2}))$, then $f$ is called a \textit{beaks} at $p$.
	\end{enumerate}
\end{Def}

For the map germ $f \colon (\R^{2},a) \to (\R^{3},f(a))$ which is right-left equivalent to the cuspidal edge, swallowtail or butterfly, the point $a \in \R^{2}$ is a non-degenerate singular point. On the other hand, if $f \colon (\R^{2},a) \to (\R^{3},f(a))$ which is right-left equivalent to the beaks, then $a \in \R^{2}$ is a degenerate singular point.

\begin{Def}
	Let $f \colon U \to \R^{3}$ be a wave front, and $p \in U$ be a non-degenerate singular point of $f$. Moreover, let $\gamma \colon (-\varepsilon,\varepsilon) \to U$ be a singular curve, and $\eta$ be a null vector field along $\gamma$. Then, a vector filed $\tilde{\eta}$ defined on $U$ such that 
	\[
		\tilde{\eta}(\gamma(t))=\eta(t)\ (\forall t \in (-\varepsilon,\varepsilon))
	\]
	is called an \textit{extended null vector field}.
\end{Def}

We next explain the criteria for singularities of wave fronts used in this section. The following theorem is a criterion for the cuspidal edge, swallowtail and butterfly.

\begin{thm}\label{thm:singular-criterion}
	Let $f \colon U \to \R^{3}$ be a wave front, and $p \in U$ be a non-degenerate singular point of $f$. Moreover, let $\hat{\lambda}$ be a singular identifier and $\tilde{\eta}$ be an extended null vector field.   Then the following holds:
	\begin{enumerate}[\normalfont (1)]
		\item The map germ $f \colon (U,p) \to (\R^{3},f(p))$ is a cuspidal edge at $p$ if and only if the following holds:
		\[
			\hat{\lambda}_{\eta}(p) \ne0,
		\]
		where $\hat{\lambda}_{\eta}:=d\hat{\lambda}(\tilde{\eta})$
		\item 
		The map germ $f \colon (U,p) \to (\R^{3},f(p))$ is a cuspidal edge at $p$ if and only if the following holds:
		\[
			\hat{\lambda}_{\eta}(p)=0,\ \hat{\lambda}_{\eta \eta}(p) \ne0,
		\]
		where $\hat{\lambda}_{\eta \eta}:=d\hat{\lambda}_{\eta}(\tilde{\eta})$
		\item The map germ $f \colon (U,p) \to (\R^{3},f(p))$ is a butterfly at $p$ if and only if the following holds:
		\[
			\hat{\lambda}_{\eta}(p)=\hat{\lambda}_{\eta \eta}(p)=0,\ \hat{\lambda}_{\eta \eta \eta}(p) \ne 0,
		\]
		where $\hat{\lambda}_{\eta \eta \eta}:=d\hat{\lambda}_{\eta \eta}(\tilde{\eta})$.
		\end{enumerate}
\end{thm}

The next theorem is a criterion for beaks.

\begin{thm}\label{thm:singular-criterion-beaks}
	Let $f \colon U \to \R^{3}$ be a wave front, and $p \in U$ be a singular point of $f$ such that $\rank(f_{u}(p),f_{v}(p))=1$. Moreover, let $\hat{\lambda}$ be a singular identifier and $\tilde{\eta}$ be an extended null vector field. Then the map germ $f \colon (U,p) \to (\R^{3},f(p))$ is a beaks at $p$ is and only if the following holds:
	\[
		d\hat{\lambda}(p)=0,\hspace{3mm} \det H_{\hat{\lambda}}(p)<0,\hspace{3mm} \hat{\lambda}_{\eta \eta}(p)\ne 0,
	\]
	where $H_{\hat{\lambda}}$ is the Hessian of $\hat{\lambda}$, that is
	\[
		H_{\hat{\lambda}}(p)=\left(
			\begin{array}{cc}
				\hat{\lambda}_{uu}(p) & \hat{\lambda}_{uv}(p) \\
				\hat{\lambda}_{vu}(p) & \hat{\lambda}_{vv}(p)\\
			\end{array}
		\right).
	\]
\end{thm}

This completes the explanation of the terminology and facts of singularity theory used in this paper. At the end of this subsection, we consider the relationship between geometric singular solutions of (\ref{ODS1}) and wave fronts. Let us recall that the GMAS $\calI$ is generated by the natural contact form on $PT^{\ast}J^{0}(2,1)$ and $\Psi$. Therefore, all two-dimensional integral manifolds of $\calI$ are also Legendrian submanifolds in $J^{1}(2,1)$. Hence we have the following proposition. 

\begin{prop}\label{prop:Legendre-wavefront}
	The projection of a two-dimensional integral manifold of a GMAS defined on $J^{1}(2,1)$ by $\varpi \colon J^{1}(2,1) \to J^{0}(2,1)$ is a wave front. 
\end{prop}

By Proposition~\ref{prop:Legendre-wavefront}, the projection of a geometric singular solution of the GMAE (\ref{ODS1}) is a wave front.

\subsection{Case of generic type}

In this subsection, we construct geometric singular solutions of (\ref{ODS1}) which is of generic type. Then, by Proposition~\ref{prop:form-I/ChI-1}, the reduced EDS $\tilde{\calI}$ is written by
\[
	\tilde{\calI}=\{dx_{2}+x_{1}dx_{4},dx_{3}-\overline{\alpha}dx_{4}\}_{\mathrm{diff}},\hspace{2mm}\overline{\alpha}:=\alpha \circ \Phi^{-1}_{1}.
\]
We here explain the construction of geometric singular solutions of (\ref{ODS1}) by using the method of Cauchy characteristics. Let $C$ be an integral curve of $\tilde{\calI}$, and 
\[
	C=\{(x_{1}(t),x_{2}(t),x_{3}(t),x_{4}(t))\ |\ t\in I\}
\]
be a local coordinate expression of $C$, where $I$ is some open interval in $\R$. Since $C$ is a one-dimensional submanifold in $U_{1} (\subset J^{1}(2,1)/\Ch(\calI))$, we can take any of $x_{i}(t)\ (i=1,\ldots,4)$ as an independent variable $t$ by the implicit function theorem. We assume $x_{4}(t)=t$. Then, since $C$ is an integral curve of $\tilde{\calI}$, the functions $x_{i}(t)\ (i=1,2,3)$ must satisfy the following system of ordinary differential equations:
\renewcommand{\arraystretch}{2}
\begin{flalign}\label{ODE1}
	\left\{
		\begin{array}{l}
			\displaystyle \frac{dx_{2}}{dt} + x_{1}(t)=0 \\
			\displaystyle \frac{dx_{3}}{dt} -\overline{\alpha}(x_{1}(t),x_{2}(t),x_{3}(t),t) =0.
		\end{array}
	\right.
\end{flalign}
\renewcommand{\arraystretch}{1}

\noindent
Here, let $\xi(t)$ be an arbitrary function, and we put 
$x_{2}:=\xi(t)$. By the first formula of (\ref{ODE1}), we have $x_{1}=-\xi'(t)$. Moreover, let $\mu(t)$ be a solution of the second formula of (\ref{ODE1}) under some initial conditions. Therefore we have the following proposition.

\begin{prop}\label{prop:singular-solution1}
	Let $\xi(t)$ be an arbitrary function, and $\mu(t)$ be a solution of the following ordinary differential equation:
	\begin{flalign}\label{ODE-eta1}
		\frac{dx_{3}}{dt} -\overline{\alpha}(-\xi'(t),\xi(t),x_{3}(t),t) =0.
	\end{flalign}
	Then the following manifold
	\begin{flalign*}
		&S_{1}:=\{
			(s,\varphi_{1}(s,-\xi'(t),\xi(t),\mu(t),t),\\
			&\hspace{30mm}\psi_{1}(s,-\xi'(t),\xi(t),\mu(t),t),\mu(t),t)\ |\ s\in I_{1},t \in J_{1}
			\} \subset J^{1}(2,1)
	\end{flalign*}
	is a two-dimensional integral manifold of $\calI$. Here $I_{1}$ and $J_{1}$ are some open intervals in $\R$.
\end{prop}

We recall that $\varphi_{1}$ and $\psi_{1}$ are second and third components of $\Phi^{-1}_{1}$, respectively. Hereafter, we consider singular points of the projection of $S_{1}$ to $J^{0}(2,1)$. Here, we define a map $H_{1} \colon S_{1} \to I_{1} \times J_{1}$ by
\[
	(s,\varphi_{1}(s,-\xi'(t),\xi(t),\mu(t),t),\psi_{1}(s,-\xi'(t),\xi(t),\mu(t),t),\mu(t),t) \mapsto (s,t).
\]
	Namely, the map $H_{1}$ is a local coordinate system of $S_{1}$.
	
\begin{nota}\label{nota:tilde}
	For technical reasons in studying singularities, we introduce the following notation: let $f$ be a function defined on $V_{1}( \subset \R \times J^{1}(2,1)/\Ch(\calI))$. Then we put
	\[
		\widetilde{f}:=f \circ \Phi_{1}|_{S_{1}} \circ H_{1}^{-1} \colon I_{1}\times J_{1} \to \R.
	\]
	Roughly speaking, the function $\widetilde{f}$ is the local coordinate expression of the map which restricts $f$ to $S_{1}$.
\end{nota}

\begin{rem}
	Let $g$ be a function defined on $\Phi^{-1}_{1}(V_{1}) (\subset J^{1}(2,1))$. In Section~4, we defined the notation $\overline{g}$. Here, we remark that the following holds
	\begin{flalign*}
		\widetilde{\overline{g}}=\overline{g} \circ \Phi_{1}|_{S_{1}} \circ H_{1}^{-1}= g \circ \Phi^{-1}_{1} \circ \Phi_{1}|_{S_{1}} \circ H_{1}^{-1} =g|_{S_{1}} \circ H^{-1}_{1}.
	\end{flalign*}
	See the following diagram for detail.
	\begin{center}
		\begin{tikzpicture}[auto]
			\node (Phi) at (0, 0) {$\Phi^{-1}_{1}(V_{1})$}; 
			\node (V1) at (2, 0) {$V_{1}$}; 
			\node (S) at (-2, 0) {$S_{1}$};
			\node (I1J1) at (-4, 0) {$I_{1}\times J_{1}$};
			\node (R) at (0, -1.5) {$\R$}; 
			\draw[->] (V1) to node [swap]{$\scriptstyle \Phi^{-1}_{1}$} (Phi);
			\draw[->] (Phi) to node {$\scriptstyle g$} (R);
			\draw[{Hooks[right]}->] (S) to node {} (Phi);
			\draw[->] (I1J1) to node {$\scriptstyle H^{-1}_{1}$} (S);
			\draw[->] (I1J1) to node [swap]{$\scriptstyle \widetilde{\overline{g}}$} (R);
			\draw[->] (V1) to node {$\scriptstyle \overline{g}$} (R);
		\end{tikzpicture}
		\end{center}
		Roughly speaking, the function $ \widetilde{\overline{g}}$ is also the local coordinate expression of the function which restricts $g$ to $S_{1}$.
\end{rem}

We next give a property of this notation.

\begin{lem}
	 Let $f_{1}$ and $f_{2}$ be functions defined on $V_{1}$. Then the following holds:
	\begin{flalign}\label{property-bar1}
		\widetilde{f_{1}+f_{2}}=\widetilde{f_{1}}+\widetilde{f_{2}},\hspace{5mm}\widetilde{f_{1}f_{2}}=\widetilde{f_{1}}\widetilde{f_{2}}.
	\end{flalign}
\end{lem}

\begin{proof}
	This is obvious because it is a property of the composition of mappings. 
\end{proof}

Under the above setting, now we study the singular points of $\varpi \colon S_{1}\to J^{0}(2,1)$, where $S_{1}$ is the integral manifold of the GMAS $\calI$ defined in Proposition~\ref{prop:singular-solution1}. Here we define $F_{1} \colon I_{1}\times J_{1} \to J^{0}(2,1)$ by
\begin{flalign}\label{def-F1}
	F_{1}(s,t):=(s,\widetilde{\varphi_{1}}(s,t),\widetilde{\psi_{1}}(s,t)).
\end{flalign}
The map $F_{1}$ is called \textit{the map associated with $S_{1}$}. One can easily see that $F_{1}=\varpi \circ H^{-1}_{1}$. Since $J^{0}(2,1)$ is a three-dimensional manifold, then it is locally diffeomorphic to $\R^{3}$. Therefore, the study of singularity of $\varpi \colon S_{1} \to J^{0}(2,1)$ coincides with the study of singularity of  $F_{1} \colon I_{1} \times J_{1} \to J^{0}(2,1)$ associated with $S_{1}$.

\begin{center}
		\begin{tikzpicture}[auto]
			\node (S) at (0, 0) {$S_{1}$}; 
			\node (J0) at (2, 0) {$J^{0}(2,1)$};
			\node (R3) at (4, 0) {$\R^{3}$}; 
			\node (I1J1) at (0, -1.5) {$I_{1}\times J_{1}$}; 
			\node (subR) at (1.1, -1.48) {$(\subset \R^{2})$};
			\draw[{Hooks[left]}->] (I1J1) to node {$\scriptstyle H^{-1}_{1}$} (S);
			\draw[->] (S) to node {$\scriptstyle \varpi$} (J0);
			\draw[->] (J0) to node {$\scriptstyle \sim $} (R3);
			\draw[->] (I1J1) to node [swap]{$\scriptstyle F_{1}$} (R3);
		\end{tikzpicture}
		\end{center}

In order to calculate the Jacobian of $F_{1}$, we consider the partial derivatives of $\widetilde{\varphi_{1}}$ and $\widetilde{\psi_{1}}$ with respect to the coordinate $s$ and $t$.
\begin{lem}\label{lem:calc-phi-psi-1}
	For the functions $\varphi_{1}$ and $\psi_{1}$, the following holds:
		
	\noindent
	$(1)\ (\widetilde{\varphi_{1}})_{s}=-\widetilde{\overline{\alpha}},
		\hspace{11.5mm}
		(2)\ (\widetilde{\varphi_{1}})_{t}=-\widetilde{\overline{\rho}} \left(\xi''(t)+\widetilde{\overline{x (\alpha_{q}+\alpha \alpha_{p})}}\right),
	$
	
	\vspace{2mm}
	\noindent
	$(3)\ (\widetilde{\psi_{1}})_{s}=\widetilde{\overline{p-\alpha q}},
		\hspace{6.5mm}
		(4)\ (\widetilde{\psi_{1}})_{t} = -\widetilde{\overline{\rho q}} \left(\xi''(t)+\widetilde{\overline{x (\alpha_{q}+\alpha \alpha_{p})}}\right),
	$
	
	\noindent
	where $\rho=1/(1+x(\alpha_{y}+q\alpha_{z}))$ which is the function defined in Section 4.
\end{lem}

\begin{proof}
	We have only to show (1) and (2), since one can show other cases, similarly. First of all, we remark that
	\begin{flalign*}
		\widetilde{\varphi_{1}}=\varphi_{1}(s,-\xi'(t),\xi(t),\mu(t),t),\hspace{4mm}
		\widetilde{\psi_{1}}=\psi_{1}(s,-\xi'(t),\xi(t),\mu(t),t).
	\end{flalign*}
	Since $\eta(t)$ is a solution of (\ref{ODE-eta1}), we have
	\[
		\mu'(t) = \overline{\alpha}(-\xi'(t),\xi(t),\mu(t),t)=\widetilde{\overline{\alpha}}.
	\]
	
	We show (1). By (\ref{dphi1}), we have
		\begin{flalign*}
			(\widetilde{\varphi_{1}})_{s} &=\widetilde{(\varphi_{1})_{x_{0}}}(s)_{s}+\widetilde{(\varphi_{1})_{x_{1}}}(-\xi'(t))_{s}+\widetilde{(\varphi_{1})_{x_{2}}}(\xi(t))_{s}+\widetilde{(\varphi_{1})_{x_{3}}}(\mu(t))_{s} + \widetilde{(\varphi_{1})_{4}}(t)_{s}\\
			&=-\widetilde{\overline{\alpha}},\\
			(\widetilde{\varphi_{1}})_{t} &=\widetilde{(\varphi_{1})_{x_{0}}}(s)_{t}+\widetilde{(\varphi_{1})_{x_{1}}}(-\xi'(t))_{t}+\widetilde{(\varphi_{1})_{x_{2}}}(\xi(t))_{t}+\widetilde{(\varphi_{1})_{x_{3}}}(\mu(t))_{t} + \widetilde{(\varphi_{1})_{4}}(t)_{t}\\
			&=-\widetilde{\overline{\rho}}\xi''(t)-\widetilde{\overline{\rho x \alpha_{z}}}\xi'(t)-\widetilde{\overline{\rho x (\alpha_{p}+x\alpha_{z})}} \mu'(t) -\widetilde{\overline{\rho x (\alpha_{q} + y\alpha_{z})}} \\
			&=-\widetilde{\overline{\rho}}\xi''(t)+\widetilde{\overline{\rho x \alpha_{z}(y+\alpha x)}}-\widetilde{\overline{\rho x \alpha (\alpha_{p}+x\alpha_{z})}} -\widetilde{\overline{\rho x (\alpha_{q} + y\alpha_{z})}}\\
			&=-\widetilde{\overline{\rho}} \left(\xi''(t)+\widetilde{\overline{x (\alpha_{q}+\alpha \alpha_{p})}}\right).
			\end{flalign*}
	
	On the other hand, since (3) and (4) can be proved similar to (1) and (2), which completes the proof. 
\end{proof}

\begin{prop}\label{prop:singular-set1}
	For the map $F_{1}$ associated with $S_{1}$, the following holds:
	\begin{flalign}\label{singset-F1}
		S(F_{1}) = \{(s,t)\in I_{1} \times J_{1}\ |\ \xi''(t) + \widetilde{\overline{x(\alpha_{q}+\alpha \alpha_{p})}}(s,t)=0\}.
	\end{flalign}
\end{prop}

\begin{proof}
	In order to show this proposition, we calculate the Jacobian of $F_{1}$. By (\ref{def-F1}) and Lemma~\ref{lem:calc-phi-psi-1}, we have
	\begin{flalign}\label{Jacobian-F1}
		\notag((F_{1})_{s} \hspace{2mm}(F_{1})_{t})
		&=\left(
			\begin{array}{cc}
				(s)_{s} & (s)_{t} \\
				(\widetilde{\varphi_{1}})_{s} &(\widetilde{\varphi_{1}})_{t} \\
				(\widetilde{\psi_{1}})_{s} & (\widetilde{\psi_{1}})_{t}
			\end{array}	
		\right)\\
		&= \left(
			\begin{array}{cc}
				1 & 0 \\
				-\widetilde{\overline{\alpha}} & -\widetilde{\overline{\rho}} \left(\xi''(t)+\widetilde{\overline{x (\alpha_{q}+\alpha \alpha_{p})}}\right)\\
				\widetilde{\overline{p-\alpha q}} & -\widetilde{\overline{\rho q}} \left(\xi''(t)+\widetilde{\overline{x (\alpha_{q}+\alpha \alpha_{p})}}\right)
			\end{array}
		\right).
	\end{flalign}
	By $\rho=1/(1+x(\alpha_{x}+q\alpha_{z}))$, we have
	\[
		\widetilde{\overline{\rho}}=1/(1+\widetilde{\overline{x(\alpha_{y}+q\alpha_{z})}}),
	\]
	and thus $\widetilde{\overline{\rho}}$ also does not take the value 0. Hence the point $(s,t)$ is a singular point of $F_{1}$ if and only if $(s,t)$ satisfies
	\[
		\xi''(t) + \widetilde{\overline{x(\alpha_{q}+\alpha \alpha_{p})}}(s,t)=0,
	\]
	which completes the proof. 
\end{proof}

We next calculate the singular identifier of $F_{1}$.
\begin{lem}\label{lem:singular-identifire}
	The singular identifier $\hat{\lambda}$ of $F_{1}$ is given by 
\begin{flalign}\label{sing-identifier-1}
	\hat{\lambda}(s,t):=\xi''(t) + \widetilde{\overline{x(\alpha_{q}+\alpha \alpha_{p})}}.
\end{flalign}
\end{lem}

\begin{proof}
	By (\ref{Jacobian-F1}), we have
	\[
		(F_{1})_{s} \times (F_{1})_{t} = -\widetilde{\overline{\rho}}\left(\xi''(t) + \widetilde{\overline{x(\alpha_{q}+\alpha \alpha_{p})}}\right){}^{t}(-\widetilde{\overline{p}},-\widetilde{\overline{q}},1),
	\]
	and put 
	\[
		\tilde{\nu}(s,t):=\frac{1}{\sqrt{(\widetilde{\overline{p}})^{2}+(\widetilde{\overline{q}})^{2}+1}}{}^{t}(-\widetilde{\overline{p}},-\widetilde{\overline{q}},1).
	\]
	Then $\tilde{\nu}$ is a unit normal vector field of $F_{1}$, and one has
	\begin{flalign*}
		\lambda(s,t) &= \det ((F_{1})_{s},(F_{1})_{t},\tilde{\nu})\\
		&=-\frac{\widetilde{\overline{\rho}}\left(\xi''(t) + \widetilde{\overline{x(\alpha_{q}+\alpha \alpha_{p})}}\right)}{\sqrt{(\widetilde{\overline{p}})^{2}+(\widetilde{\overline{q}})^{2}+1}}
		\det \left(
			\begin{array}{ccc}
				1 & 0  & -\widetilde{\overline{p}} \\
				-\widetilde{\overline{\alpha}} & 1 & -\widetilde{\overline{q}} \\
				\widetilde{\overline{p-\alpha q}} & \widetilde{\overline{q}} & 1		\end{array}
			\right)\\
			&=-\widetilde{\overline{\rho}}\sqrt{(\widetilde{\overline{p}})^{2}+(\widetilde{\overline{q}})^{2}+1}\left(\xi''(t) + \widetilde{\overline{x(\alpha_{q}+\alpha \alpha_{p})}}\right).
	\end{flalign*}
	Since the function $\widetilde{\overline{\rho}}\sqrt{(\widetilde{\overline{p}})^{2}+(\widetilde{\overline{q}})^{2}+1}$ does not take the value 0 and Proposition~\ref{prop:singular-set1}, the singularity identifier is given by
	\[
		\hat{\lambda}(s,t):=\xi''(t) + \widetilde{\overline{x(\alpha_{q}+\alpha \alpha_{p})}},
	\]
	which completes the proof.
\end{proof}

\begin{lem}\label{lem:degenerate-1}
	 The point $(s_{0},t_{0}) \in S(F_{1})$ is non-degenerate if and only if the following holds:
	\[
				\left(\widetilde{\overline{x(\alpha_{q}+\alpha \alpha_{p})}}\right)_{s} (s_{0},t_{0})\ne 0, \hspace{3mm}\text{or}\hspace{3mm}
				\xi'''(t_{0})+\left(\widetilde{\overline{x(\alpha_{q}+\alpha \alpha_{p})}}\right)_{t}(s_{0},t_{0}) \ne 0.
	\]
\end{lem}

\begin{proof}
	The singular point $(s_{0},t_{0})$ is non-degenerate if and only if
	\[
		(d\hat{\lambda})_{(s_{0},t_{0})}=\hat{\lambda}_{s}(s_{0},t_{0})(ds)_{(s_{0},t_{0})} + \hat{\lambda}_{t}(s_{0},t_{0})(dt)_{(s_{0},t_{0})} \ne 0
	\]
	Since the singularity identifier $\hat{\lambda}$ is defined by (\ref{sing-identifier-1}), this completes the proof. 
\end{proof}

We next show the criteria for the singularities of $F_{1}$.  

\begin{thm}\label{thm:singular-criterion1}
	If the point $(s_{0},t_{0}) \in S(F_{1})$ satisfies
	\[
		\xi'''(t_{0})+\left(\widetilde{\overline{x(\alpha_{q}+\alpha \alpha_{p})}}\right)_{t}(s_{0},t_{0})\ne 0,
	\]
	then $F_{1}$ is a cuspidal edge at $(s_{0},t_{0})$.
\end{thm}

\begin{proof}
	By Lemma~\ref{lem:degenerate-1}, the point $(s_{0},t_{0})$ is non-degenerate. In order to show this theorem, we use Theorem~\ref{thm:singular-criterion}~(1). We first define an extended null vector field. By Proposition~\ref{prop:singular-set1}, the point $(s,t) \in S(F_{1})$ satisfies
	\[
		\xi''(t) + \widetilde{\overline{x(\alpha_{q}+\alpha \alpha_{p})}}(s,t)=0.
	\]
	By (\ref{Jacobian-F1}), we have
	\begin{flalign}
		\notag dF_{1}|_{S(F_{1})} &=(F_{1})_{s} (ds)|_{S(F_{1})} + (F_{1})_{t}(dt)|_{S(F_{1})}\\
		\label{null-vector-1}&= {}^{t}(1,-\widetilde{\overline{\alpha}},\widetilde{\overline{p-\alpha q}})(ds)|_{S(F_{1})} +{}^{t}(0,0,0)(dt)|_{S(F_{1})}.
	\end{flalign}
	and define the vector field $\tilde{\eta}$ on $I_{1} \times J_{1}$ by
	\[
		\tilde{\eta}:=\DD/ \DD t.
	\]
	Then, by (\ref{null-vector-1}), the vector field that $\tilde{\eta}$ restricts to the singular curve is a null vector field. Hence, $\tilde{\eta}$ is an extended null vector field of $F_{1}$. This yields
	\[
		\hat{\lambda}_{\eta} =d\hat{\lambda}(\tilde{\eta})=\hat{\lambda}_{t}=\xi'''(t) + \left(\widetilde{\overline{x(\alpha_{q}+\alpha \alpha_{p})}}\right)_{t},
	\]
	and thus we have $\hat{\lambda}_{\eta}(s_{0},t_{0}) \ne0$. Hence, by Theorem~\ref{thm:singular-criterion} (1), the map $F_{1}$ is a cuspidal edge at $(s_{0},t_{0})$, which completes the proof. 
\end{proof}

\begin{thm}\label{thm:singularsol-criterion1}
	Assume that the singular point $(s_{0},t_{0}) \in \Sing (F_{1})$ satisfies
	\[
		\left(\widetilde{\overline{x(\alpha_{q}+\alpha \alpha_{p})}}\right)_{s} (s_{0},t_{0})\ne0.
	\]
	Then the following holds:
	\begin{enumerate}[\normalfont (1)]
		\setlength{\parskip}{1mm} 
  		\setlength{\itemsep}{2mm} 
		\item The function $F_{1}$ is a cuspidal edge at $(s_{0},t_{0})$ if and only if $(s_{0},t_{0})$ satisfies
		\[
			\xi'''(t_{0}) + \left(\widetilde{\overline{x(\alpha_{q}+\alpha \alpha_{p})}}\right)_{t}(s_{0},t_{0}) \ne 0.
		\]
		\item The function $F_{1}$ is a swallowtail at $(s_{0},t_{0})$ if and only if $(s_{0},t_{0})$ satisfies
			\begin{flalign*}
					\xi'''(t_{0}) + \left(\widetilde{\overline{x(\alpha_{q}+\alpha \alpha_{p})}}\right)_{t}(s_{0},t_{0}) &= 0, \\
					\xi^{(4)}(t_{0})+\left(\widetilde{\overline{x(\alpha_{q}+\alpha \alpha_{p})}}\right)_{tt}(s_{0},t_{0}) &\ne0.
			\end{flalign*}
		\item The function $F_{1}$ is a butterfly at $(s_{0},t_{0})$ if and only if $(s_{0},t_{0})$ satisfies
			\begin{flalign*}
					\xi'''(t_{0}) + \left(\widetilde{\overline{x(\alpha_{q}+\alpha \alpha_{p})}}\right)_{t}(s_{0},t_{0}) &= 0, \\
					\xi^{(4)}(t_{0})+\left(\widetilde{\overline{x(\alpha_{q}+\alpha \alpha_{p})}}\right)_{tt}(s_{0},t_{0}) &=0,\\
					\xi^{(5)}(t_{0})+\left(\widetilde{\overline{x(\alpha_{q}+\alpha \alpha_{p})}}\right)_{ttt}(s_{0},t_{0}) &\ne 0.
			\end{flalign*}
	\end{enumerate}
\end{thm}

\begin{proof}
	In order to show this theorem, we use Theorem~\ref{thm:singular-criterion}. By Lemma~\ref{lem:degenerate-1}, the singular point $(s_{0},t_{0})$ is non-degenerate. Similar to the proof of Theorem~\ref{thm:singular-criterion1}, an extended null vector field $\tilde{\eta}$ is defined by
	\[
		\tilde{\eta}:=\DD/\DD t.
	\]
	Then, by direct calculations, we have
	\begin{flalign*}
		\hat{\lambda}_{\eta} &=d\hat{\lambda}(\tilde{\eta})=\hat{\lambda}_{t}=\xi'''(t) + \left(\widetilde{\overline{x(\alpha_{q}+\alpha \alpha_{p})}}\right)_{t}, \\
		\hat{\lambda}_{\eta \eta} &= d\hat{\lambda}_{\eta}(\tilde{\eta})=\hat{\lambda}_{tt}=\xi^{(4)}(t)+\left(\widetilde{\overline{x(\alpha_{q}+\alpha \alpha_{p})}}\right)_{tt},\\
					\hat{\lambda}_{\eta \eta \eta}&=d\hat{\lambda}_{\eta \eta}(\tilde{\eta})=\hat{\lambda}_{ttt}=\xi^{(5)}(t)+\left(\widetilde{\overline{x(\alpha_{q}+\alpha \alpha_{p})}}\right)_{ttt}.
	\end{flalign*}
	Hence, by Theorem~\ref{thm:singular-criterion}, our assertion is proved. 
\end{proof}

\begin{ex}\label{ex:nondegenerate-singularsol1}
	We consider the case of $\alpha=q$. Then, by Example~\ref{ex:structure-1}, the reduced EDS $\tilde{\calI}$ is written by
		\[
			\tilde{\calI}=\{dx_{2}+x_{1}dx_{4},dx_{3}-x_{4}dx_{4}\}_{\mathrm{diff}}.
		\]
		On the other hand, the locally diffeomorphism $\Phi^{-1}_{1} \colon V_{1} \to \Phi^{-1}_{1}(V_{1})$ is given by
		\[
			\Phi^{-1}_{1}(x_{0},x_{1},x_{2},x_{3},x_{4})=(x_{0},x_{1}-x_{4}x_{0},x_{2}+x_{3}x_{0}+x_{4}(x_{1}-x_{4}x_{0}),x_{3},x_{4}).
		\]
		Therefore, we have
	\begin{flalign*}
		\varphi_{1}(x_{0},x_{1},x_{2},x_{3},x_{4})&=x_{1}-x_{4}x_{0},\\
		\psi_{1}(x_{0},x_{1},x_{2},x_{3},x_{4})&=x_{2}+x_{3}x_{0}+x_{4}(x_{1}-x_{4}x_{0}).
	\end{flalign*}
	We next take an integral curve $C=\{(x_{1}(t),x_{2}(t),x_{3}(t),t)\ |\ t \in I\}$ of $\tilde{\calI}$. Then, by solving the differential equation $dx_{3}/dt-t=0$, we have the solution
	\[
		\mu(t)=(1/2)t^{2} +D, \hspace{3mm} \text{ where $D$ is a constant of integration}.
		\]
	We put $x_{2}(t)=t^{n}\ (n \geq 3)$. Then we obtain the following integral manifold:
	\[
		S_{1}=\{(s, t^{n}-st, -nt^{n-1}+s((1/2)t^{2}+D)+t(t^{n}-st),(1/2)t^{2} +D,t)\ |\ s \in I_{1},\ t \in J_{1}\}.
	\]
	Remark that the map $F_{1} \colon I_{1} \times J_{1} \to J^{0}(2,1)$ is given by
	\[
		F_{1}(s,t)={}^{t}(s, t^{n}-st, -nt^{n-1}+s((1/2)t^{2}+D)+t(t^{n}-st)).
	\]
	We next consider singular points of $F_{1}$. By a direct calculation, we have
	\[
		\xi''(t)+\widetilde{\overline{x(\alpha_{q}+\alpha \alpha_{p})}}=n(n-1)t^{n-2}+s.
	\]
	Therefore, we obtain $(s_{0},t_{0}):=(0,0) \in S(F_{1})$. Moreover, we obtain
	\[
		\left(
			\widetilde{\overline{x(\alpha_{q}+\alpha \alpha_{p})}}
		\right)_{s} = (s)_{s}=1 \ne 0,
	\]
	and thus, one has
	\begin{flalign*}
		\xi'''(t) + \left(\widetilde{\overline{x(\alpha_{q}+\alpha \alpha_{p})}}\right)_{t} &= \frac{n!}{(n-3)!}t^{n-3}, \\
					\xi^{(4)}(t_{0})+\left(\widetilde{\overline{x(\alpha_{q}+\alpha \alpha_{p})}}\right)_{tt}(s_{0},t_{0}) &=\frac{n!}{(n-4)!}t^{n-4},\\
					\xi^{(5)}(t_{0})+\left(\widetilde{\overline{x(\alpha_{q}+\alpha \alpha_{p})}}\right)_{ttt}(s_{0},t_{0}) & = \frac{n!}{(n-5)!}t^{n-5}.
	\end{flalign*}
	Then, by Theorem~\ref{thm:singularsol-criterion1}, if $n=3$ then $F_{1}$ is the cuspidal edge at $(0,0)$. Moreover, at the point $(0,0)$, the function $F_{1}$ is the swallowtail if $n=4$, and the butterfly if $n=5$.
\end{ex}

\begin{ex}
	We next give an example of a geometric singular solution of GMAE (\ref{ODS1}) which has a degenerate singular point. In this example,  we consider the case of $\alpha = p+q^{2}$. Then, by  Example~\ref{ex:structure-1}, we have
		\[
			\tilde{\calI}=\{dx_{2}+x_{1}dx_{4},dx_{3}-(x_{3}+x_{4}^{2})dx_{4}\}_{\mathrm{diff}}.
		\]
		On the other hand, the locally diffeomorphism $\Phi^{-1}_{1} \colon V_{1} \to \Phi^{-1}_{1}(V_{1})$ is given by
		\[
			\Phi^{-1}_{1}(x_{0},x_{1},x_{2},x_{3},x_{4})=(x_{0},x_{1}-(x_{3}+x_{4}^{2})x_{0},x_{2}+x_{3}x_{0}+x_{4}(x_{1}-(x_{3}+x_{4}^{2})x_{0}),x_{3},x_{4}).
		\]
		Therefore, we have
	\begin{flalign*}
		\varphi_{1}(x_{0},x_{1},x_{2},x_{3},x_{4})&=x_{1}-(x_{3}+x_{4}^{2})x_{0},\\
		\psi_{1}(x_{0},x_{1},x_{2},x_{3},x_{4})&=x_{2}+x_{3}x_{0}+x_{4}(x_{1}-(x_{3}+x_{4}^{2})x_{0}).
	\end{flalign*}
	We next take an integral curve $C=\{(x_{1}(t),x_{2}(t),x_{3}(t),t)\ |\ t \in I\}$ of $\tilde{\calI}$. Then, by solving the differential equation $dx_{3}/dt-x_{3}(t)-t^{2}=0$, we have the solution
	\[
		\mu(t)=De^{t}-t^{2}-2t-2,\hspace{3mm} \text{ where $D$ is a constant of integration.}
	\]
	For simplicity, we consider the $D=1$ case. Now we put $x_{2}(t):=(t-\log2)^{4}$. Then, by a direct calculation, we have
	\[
		\xi''(t)+\widetilde{\overline{x(\alpha_{q}+\alpha \alpha_{p})}}=12(t-\log 2)^{2} + s(e^{t}-2).
	\]
	This yields that $(s_{0},t_{0}):=(0,\log 2) \in S(F_{1})$. Moreover, we have
		\begin{flalign*}
				\left(\widetilde{\overline{x(\alpha_{q}+\alpha \alpha_{p})}}\right)_{s}(0,\log 2)&=e^{\log 2}-2=0,\\
				\xi'''+\left(\widetilde{\overline{x(\alpha_{q}+\alpha \alpha_{p})}}\right)_{t}(0,\log 2)&=24(\log2-\log 2)+0\cdot(e^{\log2}-2)=0.
	\end{flalign*}
	Therefore, by Lemma~\ref{lem:degenerate-1}, $(0,\log2)$ is degenerate. Moreover, by Lemma~\ref{lem:singular-identifire}, the singularity identifier $\hat{\lambda}(s,t)$ is given by
	\[
		\hat{\lambda}(s,t)=12(t-\log2)^{2}+s(e^{t}-2).
	\]
	Hence, we have
	\[
		\det H_{\hat{\lambda}}=\det \left(
			\begin{array}{cc}
				\hat{\lambda}_{ss}&\hat{\lambda}_{st} \\
				\hat{\lambda}_{ts}& \hat{\lambda}_{tt} \\
			\end{array}
		\right) = \det \left(
			\begin{array}{cc}
			0 & e^{t}\\
			e^{t} & 24+se^{t}
			\end{array}
		\right)=-e^{2t},
	\]
	and thus $\det H_{\hat{\lambda}}(0,\log2)=-4<0$. In this case as well, $\tilde{\eta}=\DD/\DD t$ is also an extended null vector field. Therefore, we have
	\[
		\hat{\lambda}_{\eta \eta}=\hat{\lambda}_{tt}=24+se^{t},
	\]
	and obtain $\hat{\lambda}_{\eta \eta}(0,\log2)=24 \ne0$. Hence, by Theorem~\ref{thm:singular-criterion-beaks}, the map $F_{1}$ is the beaks at $(0,\log2)$.
\end{ex}

\subsection{Case of non-generic type}
In this subsection, we construct geometric singular solutions of (\ref{ODS1}) which is of non-generic type. Then, by Proposition~\ref{prop:form-I/ChI-2}, the reduced EDS $\tilde{\calI}$ is written by
\[
	\tilde{\calI}=\{dx_{2}+\overline{\beta}dx_{4},dx_{3}-x_{1}dx_{4}\}_{\mathrm{diff}},\hspace{3mm}\text{where }\beta:=y+\alpha x.
\]
Now we construct geometric singular solution of (\ref{ODS1}). Let $C$ be an integral curve of $\tilde{\calI}$, and 
\[
	C=\{(x_{1}(t),x_{2}(t),x_{3}(t),x_{4}(t))\ |\ t \in I\}
\] 
be a local coordinate expression, where $I$ is some open interval in $\R$. In the same way as in the subsection 6.2, we assume that $x_{4}(t)=t$. Then, since $C$ is an integral curve of $\tilde{\calI}$, the functions $x_{i}(t)\ (i=1,2,3)$ must satisfy the following system of ordinary differential equations:
\begin{flalign}\label{ODE2}
	\left\{
	\renewcommand{\arraystretch}{2}
		\begin{array}{l}
			\displaystyle \frac{dx_{2}}{dt}+\overline{\beta}(x_{1}(t),x_{2}(t),x_{3}(t),t)=0 \\
			\displaystyle \frac{dx_{3}}{dt} -x_{1}(t) = 0.
		\end{array}
	\right.
	\renewcommand{\arraystretch}{1}
\end{flalign}
Let $\xi(t)$ be an arbitrary function, and we put $x_{3}:=\xi(t)$. By the second formula of (\ref{ODE2}), we have $x_{1}(t)=\xi'(t)$. Moreover, let $\mu(t)$ be a solution of the first formula of (\ref{ODE2}) under some initial conditions. Then we have the following proposition.
\begin{prop}\label{prop:singular-solution2}
	Let $\xi(t)$ be an arbitrary function, and $\eta(t)$ be a solution of the following differential equation:
	\begin{flalign}\label{ODE-eta2}
		\frac{dx_{2}}{dt}+\overline{\beta}(\xi'(t),x_{2}(t),\xi(t),t)=0.
	\end{flalign}
	Then the following manifold
	\begin{flalign*}
		&S_{2}:=\{
		(s,\varphi_{2}(s,\xi'(t),\mu(t),\xi(t),t),\\
		&\hspace{30mm}\psi_{2}(s,\xi'(t),\mu(t),\xi(t),t),\xi(t),t)\ |\ s\in I_{2},t \in J_{2}
		\} \subset J^{1}(2,1)
	\end{flalign*}
	is a two-dimensional integral manifold of $\calI$. Here $I_{2}$ and $J_{2}$ are some open intervals in $\R$.
\end{prop}

We recall that $\varphi_{2}$ and $\psi_{2}$ are second and third components of $\Phi^{-1}_{2}$, respectively. Hereafter, we consider singular points of the projection $S_{1}$ to $J^{0}(2,1)$. Here, we define the map $H_{1}\colon S_{2} \to I_{2} \times J_{2}$ by
\[
	H_{2}(s,\varphi_{2}(s,\xi'(t),\mu(t),\xi(t),t),\psi_{2}(s,\xi'(t),\mu(t),\xi(t),t),\xi(t),t) :=(s,t).
\]
Namely, the map $H_{2}$ is a local coordinate system of $S_{2}$.

\begin{nota}
	We define the following notation in the same way as in the previous subsection: let $f$ be a function defined on $V_{2}(\subset \R \times J^{1}(2,1)/\Ch(\calI))$. Then we put
	\[
		\widetilde{f}:=f \circ \Phi_{2}|_{S_{2}} \circ H_{2}^{-1} \colon I_{2} \times J_{2} \to \R.
	\]
	Note that we use the same notation in Notation~\ref{nota:tilde} without fear of confusion.
\end{nota}

\begin{rem}
	Let $g$ be a function defined on $\Phi^{-1}_{2}(V_{2}) (\subset J^{1}(2,1))$. In Section~4, we defined the notation $\overline{g}$. Here, we remark that the following holds:
	\begin{flalign*}
		\widetilde{\overline{g}}=\overline{g} \circ \Phi_{2}|_{S_{2}} \circ H_{2}^{-1}= g \circ \Phi^{-1}_{2} \circ \Phi_{2}|_{S_{2}} \circ H_{2}^{-1} =g|_{S_{2}} \circ H^{-1}_{2}.
	\end{flalign*}
\end{rem}

As in the previous subsection, we have the following lemma.

\begin{lem}
	 Let $f_{1}$ and $f_{2}$ be functions defined on $V_{2}$. Then the following holds:
	\begin{flalign}\label{property-bar1}
		\widetilde{f_{1}+f_{2}}=\widetilde{f_{1}}+\widetilde{f_{2}},\hspace{5mm}\widetilde{f_{1}f_{2}}=\widetilde{f_{1}}\widetilde{f_{2}}.
	\end{flalign}
\end{lem}

Under the above setting, now we study the singular points of $\varpi \colon S_{2}\to J^{0}(2,1)$, where $S_{2}$ is the integral manifold of the GMAS $\calI$ defined in Proposition~\ref{prop:singular-solution2}. Here we define $F_{2} \colon I_{2}\times J_{2} \to J^{0}(2,1)$ by
\[
	F_{2}(s,t):=(s,\widetilde{\varphi_{2}}(s,t),\widetilde{\psi_{2}}(s,t)).
\]
The map $F_{2}$ is called \textit{the map associated with $S_{2}$}. One can easily see that $F_{1}=\varpi \circ H^{-1}_{2}$. As in the previous subsection, the study of singularity of $\varpi \colon S_{2} \to J^{0}(2,1)$ coincides with the study of singularity of $F_{2} \colon I_{2} \times J_{2} \to J^{0}(2,1)$ associated with $S_{2}$.

In order to calculate the Jacobian of $F_{2}$, we consider the partial derivatives of $\widetilde{\varphi_{2}}$ and $\widetilde{\psi_{2}}$ by the coordinate $(s,t)$ on $I_{2} \times J_{2}$.

\begin{lem}\label{lem:calc-phi-psi-2}
	For the functions $\widetilde{\varphi_{2}}$ and $\widetilde{\psi_{2}}$, the following holds
	
	\vspace{2mm}	
	\noindent
	$
		(1)\ (\widetilde{\varphi_{2}})_{s}=\widetilde{\overline{x(\alpha_{x}+p\alpha_{z}}}),
		\hspace{13.5mm}
		(2)\ (\widetilde{\varphi_{2}})_{t}=\widetilde{\overline{x}}\left(\widetilde{\overline{\alpha_{q}+\alpha \alpha_{p}}}-\xi''(t)\right),
	$
	
	\vspace{2mm}
	
	\noindent
	$
		(3)\ (\widetilde{\psi_{2}})_{s}=\widetilde{\overline{p+qx(\alpha_{x}+p\alpha_{z})}}, 
		\hspace{6mm}
		(4)\ (\widetilde{\psi_{2}})_{t} = \widetilde{\overline{qx}}\left(\widetilde{\overline{\alpha_{q}+\alpha \alpha_{p}}}-\xi''(t)\right).
	$
\end{lem}

\begin{proof}
	We have only to show (1) and (2), since one can show other cases, similarly. First of all, we remark that
	\begin{flalign*}
		\widetilde{\varphi_{2}}(s,t)=\varphi_{1}(s,\xi'(t),\mu(t),\xi(t),t),\hspace{4mm}
		\widetilde{\psi_{2}}(s,t)=\psi_{1}(s,\xi'(t),\mu(t),\xi(t),t).
	\end{flalign*}
	Moreover, since $\mu(t)$ is a solution of (\ref{ODE-eta2}), we have
	\[
		\mu'(t) = -\overline{\beta}(\xi'(t),\mu(t),\xi(t),t)=-\widetilde{\overline{y+\alpha x}}.
	\]
	Then, by direct calculations, one has
	\begin{flalign*}
		(\widetilde{\varphi_{2}})_{s}&=\widetilde{(\varphi_{2})_{x_{0}}}(s)_{s}+\widetilde{(\varphi_{2})_{x_{1}}}(\xi'(t))_{s}+\widetilde{(\varphi_{2})_{x_{2}}}(\mu(t))_{s}+\widetilde{(\varphi_{2})_{x_{3}}}(\xi(t))_{s} + \widetilde{(\varphi_{2})_{4}}(t)_{s}\\
		&=\widetilde{\overline{x(\alpha_{x}+p\alpha_{z})}},\\
		(\widetilde{\varphi_{2}})_{t}&=\widetilde{(\varphi_{2})_{x_{0}}}(s)_{t}+\widetilde{(\varphi_{2})_{x_{1}}}(\xi'(t))_{t}+\widetilde{(\varphi_{2})_{x_{2}}}(\mu(t))_{t}+\widetilde{(\varphi_{2})_{x_{3}}}(\xi(t))_{t} + \widetilde{(\varphi_{2})_{4}}(t)_{t}\\
		&=-\widetilde{\overline{x}}\xi''(t)+\widetilde{\overline{x\alpha_{z}}}\mu'(t) + \widetilde{\overline{x(\alpha_{p}+x\alpha_{z})}}\xi'(t)+\widetilde{\overline{x(\alpha_{q}+y\alpha_{z})}}\\
		&=-\widetilde{\overline{x}}\xi''(t)-\widetilde{\overline{x\alpha_{z}(y+\alpha x)}} + \widetilde{\overline{x(\alpha_{p}+x\alpha_{z})}}\xi'(t)+\widetilde{\overline{x(\alpha_{q}+y\alpha_{z})}}\\
		&=\widetilde{\overline{x}}\left(\widetilde{\overline{\alpha_{q}+\alpha \alpha_{p}}}-\xi''(t)\right),
	\end{flalign*}
which completes the proof.
\end{proof}

\begin{prop}\label{prop:singular-set2}
	For the map $F_{2}$ associated with $S_{2}$, the following holds:
	\[
		S(F_{2}) = \{(s,t)\in I_{2} \times J_{2}\ |\ \widetilde{\overline{\alpha_{q}+\alpha \alpha_{p}}}(s,t)-\xi''(t)=0\}.
	\]
\end{prop}

\begin{proof}
	In order to prove this proposition, it is enough to calculate the Jacobian of $F_{2}$. By Lemma~\ref{lem:calc-phi-psi-2}, we have
	{\small
	\begin{flalign}\label{Jacobian-F2}
		\notag((F_{2})_{s}\ (F_{2})_{t})&=\left(
			\begin{array}{cc}
				(s)_{s} & (s)_{t} \\
				(\widetilde{\varphi_{2}})_{s} &(\widetilde{\varphi_{2}})_{t} \\
				(\widetilde{\psi_{2}})_{s} & (\widetilde{\psi_{2}})_{t}
			\end{array}	
		\right)\\
		&= \left(
			\begin{array}{cc}
				1 & 0 \\
				\widetilde{\overline{x(\alpha_{x}+p\alpha_{z}}}) & \widetilde{\overline{x}}\left(\widetilde{\overline{\alpha_{q}+\alpha \alpha_{p}}}-\xi''(t)\right) \\
			\widetilde{\overline{p+qx(\alpha_{x}+p\alpha_{z})}} & \widetilde{\overline{qx}}\left(\widetilde{\overline{\alpha_{q}+\alpha \alpha_{p}}}-\xi''(t)\right)
			\end{array}
		\right).
	\end{flalign}
	}
	
	\noindent
	Since the function $\alpha$ satisfies $1+x(\alpha_{y}+q\alpha_{z})=0$, we have $1+\widetilde{\overline{x}}(\widetilde{\overline{\alpha_{y}+q\alpha_{z}}}) =0$. Therefore, $\widetilde{\overline{x}}$ does not also take 0. Hence the point $(s,t)$ is a singular point of $F_{2}$ if and only if $(s,t)$ satisfies
	\[
	\widetilde{\overline{\alpha_{q}+\alpha \alpha_{p}}}(s,t)-\xi''(t)=0,
	\]
	which completes the proof.
\end{proof}

We next calculate the singular identifier of $F_{2}$

\begin{lem}\label{lem:singular-identifire2}
	The singular identifier $\hat{\lambda}$ of $F_{2}$ is given by the following: 
\begin{flalign}\label{sing-identifier-2}
	\hat{\lambda}(s,t):=\widetilde{\overline{\alpha_{q}+\alpha \alpha_{p}}}(s,t)-\xi''(t).
\end{flalign}
\end{lem}

\begin{proof}
	By direct calculations, we have
	\[
		(F_{2})_{s} \times (F_{2})_{t} = \widetilde{\overline{x}}\left(\widetilde{\overline{\alpha_{q}+\alpha \alpha_{p}}}-\xi''(t)\right){}^{t}(-\widetilde{\overline{p}},-\widetilde{\overline{q}},1),
	\]
	and put
	\[
		\tilde{\nu}=\frac{1}{\sqrt{(\widetilde{\overline{p}})^{2}+(\widetilde{\overline{q}})^{2}+1}}{}^{t}(-\widetilde{\overline{p}},-\widetilde{\overline{q}},1).
	\]
	Then $\tilde{\nu}$ is a unit normal vector field of $F_{2}$. Therefore, we have
	\begin{flalign*}
	\lambda(s,t) &= \det ((F_{2})_{s},(F_{2})_{t},\widetilde{\nu})\\
	&=\frac{\widetilde{\overline{x}}\left(\widetilde{\overline{\alpha_{q}+\alpha \alpha_{p}}}-\xi''(t)\right)}{\sqrt{(\widetilde{\overline{p}})^{2}+(\widetilde{\overline{q}})^{2}+1}}
	\det \left(\begin{array}{ccc}
		1 & 0  & -\widetilde{\overline{p}}\\
				\widetilde{\overline{x(\alpha_{x}+p\alpha_{z}}}) & 1 & -\widetilde{\overline{q}} \\
			\widetilde{\overline{p+qx(\alpha_{x}+p\alpha_{z})}} & t & 1
			\end{array}
	\right)\\
	&=\widetilde{\overline{x}}\sqrt{(\widetilde{\overline{p}})^{2}+(\widetilde{\overline{q}})^{2}+1}\left(\widetilde{\overline{\alpha_{q}+\alpha \alpha_{p}}}-\xi''(t)\right).
\end{flalign*}
	Since the function $\widetilde{\overline{x}}\sqrt{(\widetilde{\overline{p}})^{2}+(\widetilde{\overline{q}})^{2}+1}$ does not take the value 0 and Proposition~\ref{prop:singular-set2}, the singularity identifier is the following:
	\[
		\hat{\lambda}(s,t):=\widetilde{\overline{\alpha_{q}+\alpha \alpha_{p}}}-\xi''(t),
	\]
	which completes the proof.
\end{proof}

\begin{lem}\label{lem:degenerate-2}
	The point $(s_{0},t_{0}) \in S(F_{2})$ is non-degenerate if and only if the following holds:
	\[
		\left(\widetilde{\overline{\alpha_{q}+\alpha \alpha_{p}}}\right)_{s} (s_{0},t_{0})\ne 0, \hspace{3mm}\text{or}\hspace{3mm}
		\left(\widetilde{\overline{\alpha_{q}+\alpha \alpha_{p}}}\right)_{t} (s_{0},t_{0}) -\xi'''(t_{0})	 \ne 0.
	\]
\end{lem}

\begin{proof}
	The singular point $(s_{0},t_{0})$ is non-degenerate if and only if
	\[
		(d\hat{\lambda})_{(s_{0},t_{0})}=\hat{\lambda}_{s}(s_{0},t_{0})(ds)_{(s_{0},t_{0})} + \hat{\lambda}_{t}(s_{0},t_{0})(dt)_{(s_{0},t_{0})} \ne 0.
	\]
	Since the singularity identifier $\hat{\lambda}$ is defined by (\ref{sing-identifier-2}), this completes the proof.
\end{proof}

We next show the criteria for the singularities of $F_{2}$.  

\begin{thm}\label{thm:singularsol-criterion2}
	If the point $(s_{0},t_{0}) \in S (F_{2})$ satisfies
	\[
		\left(\widetilde{\overline{\alpha_{q}+\alpha \alpha_{p}}}\right)_{t} (s_{0},t_{0}) -\xi'''(t_{0})\ne 0,
	\]
	then $F_{2}$ is a cuspidal edge at $(s_{0},t_{0})$.
\end{thm}

\begin{proof}
		By Lemma~\ref{lem:degenerate-2}, the point $(s_{0},t_{0})$ is non-degenerate. In order to show this theorem, we use Theorem~\ref{thm:singular-criterion} (1). We first define an extended null vector field. 
		By Proposition~\ref{prop:singular-set2}, the point $(s,t) \in S(F_{2})$ satisfies
		\[
			\widetilde{\overline{\alpha_{q}+\alpha \alpha_{p}}}(s,t)-\xi''(t)=0.
		\]
		By (\ref{Jacobian-F2}), we have
	\begin{flalign*}
		dF_{2}|_{S(F_{2})} &=(F_{2})_{s} (ds)|_{S(F_{2})} + (F_{2})_{t}(dt)|_{S(F_{2})}\\
		&= {}^{t}(1,\widetilde{\overline{x(\alpha_{x}+p\alpha_{z}}}),\widetilde{\overline{p+qx(\alpha_{x}+p\alpha_{z})}})(ds)|_{S(F_{2})} +{}^{t}(0,0,0)(dt)|_{S(F_{2})}.
	\end{flalign*}
	Therefore, as in the previous subsection, the following vector field
	\[
		\tilde{\eta}:=\DD/\DD t
	\]
	is an extended null vector field of $F_{2}$. This yields that
	\[
		\hat{\lambda}_{\eta}=d\hat{\lambda}(\tilde{\eta})=\hat{\lambda}_{t}=\left(\widetilde{\overline{\alpha_{q}+\alpha \alpha_{p}}}\right)_{t} -\xi'''(t),
	\]
	and thus, we have $\hat{\lambda}_{\eta}(s_{0},t_{0}) \ne0$. Hence, by Theorem~\ref{thm:singular-criterion} (1), the map $F_{2}$ is a cuspidal edge at $(s_{0},t_{0})$, which completes the proof.
\end{proof} 

\begin{thm}
	Assume that $(s_{0},t_{0}) \in S(F_{2})$ satisfies 
	\[
		\left(\widetilde{\overline{\alpha_{q}+\alpha \alpha_{p}}}\right)_{s} (s_{0},t_{0})\ne 0.
	\]
	Then the following holds:
	\begin{enumerate}[\normalfont (1)]
		\setlength{\parskip}{1mm} 
  		\setlength{\itemsep}{2mm} 
		\setlength{\leftskip}{5mm}
		\item The map $F_{2}$ is a cuspidal edge at $(s_{0},t_{0})$ if and only if $(s_{0},t_{0})$ satisfies
		\begin{flalign*}
			\left(\widetilde{\overline{\alpha_{q}+\alpha \alpha_{p}}}\right)_{t} (s_{0},t_{0}) -\xi'''(t_{0})\ne 0.
		\end{flalign*}
		\item The map $F_{2}$ is a swallowtail at $(s_{0},t_{0})$ if and only if $(s_{0},t_{0})$ satisfies			\begin{flalign*}
				\left(\widetilde{\overline{\alpha_{q}+\alpha \alpha_{p}}}\right)_{t} (s_{0},t_{0}) -\xi'''(t_{0}) &= 0, \\
				\left(\widetilde{\overline{\alpha_{q}+\alpha \alpha_{p}}}\right)_{tt} (s_{0},t_{0}) -\xi^{(4)}(t_{0})&\ne 0.
			\end{flalign*}
		\item The map $F_{2}$ is a butterfly at $(s_{0},t_{0})$ if and only if $(s_{0},t_{0})$ satisfies
			\begin{flalign*}
				\left(\widetilde{\overline{\alpha_{q}+\alpha \alpha_{p}}}\right)_{t} (s_{0},t_{0}) -\xi'''(t_{0}) &= 0, \\
				\left(\widetilde{\overline{\alpha_{q}+\alpha \alpha_{p}}}\right)_{tt} (s_{0},t_{0}) -\xi^{(4)}(t_{0})&= 0,\\
\left(\widetilde{\overline{\alpha_{q}+\alpha \alpha_{p}}}\right)_{ttt} (s_{0},t_{0}) -\xi^{(5)}(t_{0})&\ne 0.
			\end{flalign*}
	\end{enumerate}
\end{thm}

\begin{proof}
	In order to show this theorem, we use Theorem~\ref{thm:singular-criterion}. By our assumption and Lemma~\ref{lem:degenerate-2}, the singular point $(s_{0},t_{0})$ is non-degenerate. Similar to the proof of Theorem~\ref{thm:singularsol-criterion2}, an extended null vector field $\tilde{\eta}$ is defined by
	\[
		\tilde{\eta}:=\DD/\DD t.
	\]
	Then by direct calculations, we have
	\begin{flalign*}
		\hat{\lambda}_{\eta}&=d\hat{\lambda}(\tilde{\eta})=\hat{\lambda}_{t}=\left(\widetilde{\overline{\alpha_{q}+\alpha \alpha_{p}}}\right)_{t} -\xi'''(t), \\
		\hat{\lambda}_{\eta \eta}&=d\hat{\lambda}_{\eta}(\tilde{\eta})=\hat{\lambda}_{tt} = \left(\widetilde{\overline{\alpha_{q}+\alpha \alpha_{p}}}\right)_{tt} -\xi^{(4)}(t),\\
		\hat{\lambda}_{\eta \eta \eta }&=d\hat{\lambda}_{\eta \eta}(\tilde{\eta})=\hat{\lambda}_{ttt} = \left(\widetilde{\overline{\alpha_{q}+\alpha \alpha_{p}}}\right)_{ttt} -\xi^{(5)}(t).
	\end{flalign*}
	Hence, by Theorem~\ref{thm:singular-criterion}, our assertion is proved.
\end{proof}

\begin{ex}
	We consider the case of $\alpha=(q-y)/x$. By Proposition~\ref{prop:form-I/ChI-2}, we have
		\[
			\tilde{\calI}=\{dx_{2}+x_{4}dx_{4}, dx_{3}-x_{1}dx_{4}\}_{\mathrm{diff}}.
		\]
		On the other hand, the inverse map of $\Phi_{1}$ locally is given by
		\[
			\Phi^{-1}_{2}(x_{0},x_{1},x_{2},x_{3},x_{4})=(x_{0},x_{4}-x_{0}x_{1},x_{2}+x_{3}x_{0}+x_{4}(x_{4}-x_{0}x_{1}),x_{3},x_{4}).
		\]
		Therefore, we have
	\begin{flalign*}
		\varphi_{2}(x_{0},x_{1},x_{2},x_{3},x_{4})&=x_{4}-x_{0}x_{1},\\
		\psi_{2}(x_{0},x_{1},x_{2},x_{3},x_{4})&=x_{2}+x_{3}x_{0}+x_{4}(x_{4}-x_{0}x_{1}).
	\end{flalign*}
	We next take an integral curve $C=\{(x_{1}(t),x_{2}(t),x_{3}(t),t)\ |\ t \in I\}$ of $\tilde{\calI}$. Then, by solving the differential equation $dx_{2}/dt+t=0$, we have the solution
	\[
		\mu(t)=-(1/2)t^{2} +D,\hspace{3mm}\text{where $D$ is a constant of integration.}
	\]
	Here we put $x_{3}(t):=t^{2}+t^{n}\ (n \geq 3)$. Then by a direct calculation, we have
	\[
		\widetilde{\overline{\alpha_{q}+\alpha \alpha_{p}}}-\xi''(t)=\frac{1}{s}-2-n(n-1)t^{n-2}.
	\]
	Then the point $(s_{0},t_{0}):=(1/2,0)$ is a singular point of $F_{2}$. Moreover, we obtain
	\[
		\left(\widetilde{\overline{\alpha_{q}+\alpha \alpha_{p}}}\right)_{s}=-1/s^{2} \ne 0,
	\]
	then $(s_{0},t_{0})$ is non-degenerate. Therefore, one has
	\begin{flalign*}
		\left(\widetilde{\overline{\alpha_{q}+\alpha \alpha_{p}}}\right)_{t}  -\xi'''(t) &= -\frac{n!}{(n-3)!}t^{n-3}, \\
		\left(\widetilde{\overline{\alpha_{q}+\alpha \alpha_{p}}}\right)_{tt} -\xi^{(4)}(t)&= -\frac{n!}{(n-4)!}t^{n-4},\\
\left(\widetilde{\overline{\alpha_{q}+\alpha \alpha_{p}}}\right)_{ttt} -\xi^{(5)}(t)&= -\frac{n!}{(n-5)!}t^{n-5}.
	\end{flalign*}
	Then, by Theorem~\ref{thm:singularsol-criterion2}, if $n=3$ then $F_{1}$ is a cuspidal edge at $(1/2,0)$. Moreover, at the point $(1/2,0)$, the function $F_{1}$ is the swallowtail if $n=4$, and the butterfly if $n=5$.
\end{ex}

\section{Acknowledgement}
The author would like to thank Kazuhiro Shibuya and Koichi Takeuchi for useful comments and suggestions. The author was supported by JSPS KAKENHI Grant Number 21J12161.

\bibliographystyle{plain}
\bibliography{mybibfile}
\end{document}